\documentclass[11pt, a4paper]{article}
\usepackage[a4paper, margin=1in,inner=1in, outer=1in]{geometry}
\usepackage[authoryear, round]{natbib}
\RequirePackage{amsthm,amsmath,amsfonts,amssymb}
\RequirePackage[colorlinks,citecolor=blue,urlcolor=blue]{hyperref}
\RequirePackage{graphicx}

\usepackage{amsthm}
\usepackage{adjustbox}

\usepackage{mathtools}
\mathtoolsset{showonlyrefs=true}
\usepackage{tikz}
\usepackage{bbm}
\usepackage{enumitem}
\newcommand{\inner}[2]{\left \langle #1, #2 \right \rangle}

\newcommand{\iid}{\overset{\text{i.i.d.}}{\sim}}

\newcommand{\NN}{\mathbb{N}}
\newcommand{\PP}{\mathbb{P}}

\newcommand{\RR}{\mathbb{R}}
\newcommand{\EE}{\mathbb{E}}
\newcommand{\ZZ}{\mathbb{Z}}

\newcommand\ind{\mathbbm{1}}
\newcommand\normm[1]{\left \lVert #1 \right \rVert}
\newcommand\normmo[1]{\left \lVert #1 \right \rVert_{\Psi_2}}
\newcommand\normmotwo[1]{\left \lVert #1 \right \rVert_{\Psi_1}}
\newcommand\normmop[1]{\left \lVert #1 \right \rVert_{\mathrm{op}}}
\newcommand{\Var}{\mathrm{Var}}
\newcommand{\Cov}{\mathrm{Cov}}
\newcommand{\Tr}{\mathrm{Tr}}
\newcommand{\rank}{\mathrm{rank}}
\newcommand{\PSD}{\mathrm{PSD}}
\newcommand{\PD}{\mathrm{PD}}

\theoremstyle{plain}
\newtheorem{theorem}{Theorem}
\newtheorem{proposition}{Proposition}
\newtheorem{lemma}{Lemma}

\theoremstyle{plain}
\newtheorem{assumption}{Assumption}

\newtheorem{definition}{Definition} 
\begin{document}

\title{Minimax rates in variance and covariance changepoint testing}
\author{Per August Jarval Moen$^{1}$}

\date{$^1$Department of Mathematics, University of Oslo}

\maketitle

\begin{abstract}
We study the detection of a change in the covariance matrix of $n$ independent sub-Gaussian random variables of dimension $p$. Our first contribution is to show that $\log\log(8n)$ is the exact minimax testing rate for a change in variance when $p=1$, thereby giving a complete characterization of the problem for univariate data. Our second contribution is to derive a lower bound on the minimax testing rate under the operator norm, taking a certain notion of sparsity into account. In the low- to moderate-dimensional region of the parameter space, we are able to match the lower bound from above with an optimal test based on sparse eigenvalues. In the remaining region of the parameter space, where the dimensionality is high, the minimax lower bound implies that changepoint testing is very difficult. As our third contribution, we propose a computationally feasible variant of the optimal multivariate test for a change in covariance, which is also adaptive to the nominal noise level and the sparsity level of the change.\end{abstract}

\section{Introduction}
Changepoint analysis, the problem of detecting or locating one or more distributional changes in a data sequence, has received renewed attention during the last decade. This surge in interest can be primarily attributed to its myriad applications in conjunction with the increasing availability of data. Distributional changes are of interest in several sciences as they can be interpreted as regime shifts. Examples include Economics \citep{econometrica}, Political Science \citep{djuve_patterns_2020}, Peace Research \citep{cunen_statistical_2020} and Climate Research \citep{climateresearch}, to mention a few. Another noteworthy application of changepoint analysis is condition monitoring, such as industrial \citep{industrial_monitoring} or medical \citep{health}, where a changepoint may indicate a critical fault or abnormal behavior.  

Several theoretical contributions regarding the minimax properties of changepoint problems have been made in the past five years. Importantly, these contributions offer critical insight into the inherent difficulty of changepoint problems and serve as natural benchmarks for statistical methods. Most of the existing minimax results in the changepoint literature relate to changes in the mean vector. For instance, \cite{liu_minimax_2021} derive the exact minimax rate of testing for a single (possibly sparse) change in the mean vector with Gaussian noise, while \cite{li2023robust} derive the minimax rate of testing in the more difficult scenario when the tails of the noise are either sub-Weibull or polynomially decaying. Other worthy mentions are the works of \cite{enikeeva_highdimensional}; \cite{cusumandoptimality}; \cite{pilliat} and \cite{verzelen_optimal}, all focusing on changes in the mean vector.

The story is somewhat different for problems involving variance or covariance changes. By a covariance change, we refer to a change in the spatial covariance matrix of the data, which for univariate data is simply a change in the variance. Such problems are fundamentally distinct from mean-change problems, as a covariance change alters the noise level of the data, unlike a change in the mean. To practitioners, covariance shifts are arguably of no less interest than mean shifts, as a covariance change entails a shift in either the variability of individual time series or their inter-dependencies. In finance, for example, the covariance between assets is of critical importance for portfolio construction (see, for instance,  \citealt{finance2} and \citealt{finance3}). In neuroscience, covariance and precision matrices are used quantify the functional connectivity between different regions of the brain, such as in \cite{varoquaux2010}; \cite{cribben2013} and \cite{jeong2016}. Despite the large number of proposed methods for covariance changepoint detection, such as those of \cite{aue2009}; \cite{galeano2014} and \cite{lili2023}, the fundamental difficulty of such problems have yet to be fully explored in the changepoint literature. 

Two exceptions are the recent works of \cite{covariancecusum} and \cite{enikeeva}. The former derive a finite sample minimax lower bound on the signal strength at which no statistical method can reasonably estimate the location of a change in covariance in independent multivariate sub-Gaussian data. The lower bound does not cover the whole region of the parameter space, and does not capture the dependence on the sample size, nor any notion of sparsity. The latter paper considers a different setting with low rank VAR processes, and derives the asymptotic minimax testing rate for a change in the transition matrix of the observed process. In the sequel, we will only compare our results with those of \cite{covariancecusum}, as the setting investigated in \cite{enikeeva} is too different from ours to obtain any meaningful comparisons. To the best of the author's knowledge, no other contributions have been made to characterize the statistical hardness of changepoint problems with a covariance change. 

For stationary models, the inherent statistical hardness of variance and covariance related problems have been much more studied than for changepoint models. The works most related to ours are that of \cite{berthet_optimal} and \cite{cai_optimal_2015}. The former propose a test for a sparse perturbation in the covariance matrix of i.i.d. Gaussian data based on $s$-sparse eigenvalues, and the latter show that this test is minimax rate optimal in the region of the parameter space where the signal strength is small or moderate. The current work builds upon the ideas and techniques in these two papers, although our problem setting, being a two-sample problem, is fundamentally different from theirs. For other works, we refer to \cite{covariance1}, \cite{covariance2} and \cite{covariance3}.

In this paper, we investigate the minimax rate of testing for a change in the covariance matrix of sequence of $n\geq 2$ independent, centered, sub-Gaussian observations of dimension $p$. We first consider the univariate case where $p=1$, in which case the covariance matrix is simply the variance, under the assumption that the variance-rescaled data have bounded densities and the sub-Gaussian norm of the data is of the same order as the variance. Given a sequence of random variables $X_1, \ldots, X_n$, we wish to test the null hypothesis that the sequence of variances $\EE( X_i^2)$ is constant versus the alternative that 
\begin{align}
    \EE (X_i^2) = \begin{cases}
        \sigma_1^2, &\text{if } i \leq t_0,\\
        \sigma_2^2, &\text{if } i \geq t_0+1,
    \end{cases}
\end{align}
for some unknown changepoint location $1\leq t_0 \leq n-1$ and unknown pre- and post-change variances $\sigma_1^2, \sigma_2^2>0$. A novelty of this work is to recognize that the difficulty of this problem is governed by a signal strength parameter $\rho$ that measures the quantity
\begin{align}
    \min(t_0, n - t_0) \left\{ \left( \frac{\left|\sigma_1^2 - \sigma_2^2\right| }{\min( \sigma_1^2,\sigma_2^2)}  \right) \wedge \left(  \frac{\left|\sigma_1^2 - \sigma_2^2\right| }{\min( \sigma_1^2,\sigma_2^2)}\right)^2\right\}. 
\end{align}
Here, the symbol $\wedge$ indicates the minimum value. The first factor $\min(t_0, n-t_0)$ is the effective sample size of the problem, which is equal to the smallest of the number of observations before and after the change. The second factor measures the absolute size of the variance change relative to the smallest of the pre- and post-change variances. We remark that $|\sigma_1^2 - \sigma_2^2| / \min( \sigma_1^2,\sigma_2^2)$ can equivalently be written as the variance ratio $\{(\sigma_1^2 / \sigma_2^2) \vee (\sigma_2^2 / \sigma_1^2)\} -1$. 

The minimax rate of testing, denoted $\rho^*(n)$, is a function of $n$ for which the worst-case sum of Type I and Type II errors of any test can be made arbitrarily large whenever $\rho \leq c \rho^*(n)$, and arbitrarily small for \textit{some} test whenever $\rho \geq C \rho^*(n)$, by choosing the constants $c,C>0$ appropriately. The first contribution of this paper is to show that
\begin{align}
    \rho^*(n) \asymp \log \log(8n), \label{skryt1}
\end{align}
thereby giving a complete characterization of this problem, which is done in Section \ref{univariatesec}.  As an aside, the optimal test constructed to reach the minimax rate is conceptually simple, computationally efficient and easy to implement. 

Next we consider multivariate data, assuming that the variance-rescaled data have bounded densities and the sub-Gaussian norm of the data is of the same order as the variance, now along any one-dimensional subspace of $\RR^p$. Given an independent sequence $X_1, \ldots, X_n$ of $p$-dimensional random variables, we wish to test the null hypothesis that $\EE( X_i X_i^\top)$ is constant and positive definite versus the alternative that 
\begin{align}
    \EE (X_i X_i^\top) = \begin{cases}
        \Sigma_1, &\text{if } i \leq t_0,\\
        \Sigma_2, &\text{if } i \geq t_0+1,
    \end{cases}
\end{align}
for some unknown changepoint location $1\leq t_0 \leq n-1$, where $\Sigma_1 \neq \Sigma_2$ are positive definite and unknown. The size of the change under the alternative is measured in terms of the operator norm of the matrix $\Sigma_1 - \Sigma_2$. Since $\Sigma_1 - \Sigma_2$ is symmetric, its operator norm is given by 
\begin{align}
    \lVert \Sigma_1 - \Sigma_2 \rVert_{\mathrm{op}} = \underset{v \in S^{p-1}}{\sup} \ | v^\top (\Sigma_1 - \Sigma_2 )v |, \label{operatornorm}
\end{align}
where $S^{p-1}$ denotes the Euclidean unit sphere in $\RR^p$. The right-hand side of \eqref{operatornorm} is the largest absolute change in the variance of the data along some one-dimensional axis of variation. If an axis maximizing the change in variance is spanned by a vector with at most $s$ non-zero entries, we say that the change in covariance is $s$-sparse. That is, whenever the $s$-sparse eigenvalue $\lambda_{\max}^s(\Sigma_1 - \Sigma_2)$ agrees with the operator norm of $\Sigma_1 - \Sigma_2$, defined by 
\begin{align}
    \lambda_{\max}^s(A) &= \underset{v \in S^{p-1}_s}{\sup} \ | v^\top A v | \label{firsteigsparse}
\end{align}
for any symmetric matrix $A$, where $S^{p-1}_s$ denotes the subspace of the Euclidean unit sphere containing only vectors with at most $s$ non-zero entries. As an example, the change in covariance is $s$-sparse whenever $\Sigma_2 = \Sigma_1 + v v^\top$, and $v$ has no more than $s$ non-zero entries. Similar notions of sparsity are used by e.g.~\cite{cai_optimal_2015} and \cite{berthet_optimal} for sparse PCA problems. 

As a natural extension to the univariate setting, we define the signal strength parameter $\rho$ to measure the quantity
\begin{align}
    \min(t_0, n - t_0)  \left\{ \left( \frac{\normmop{\Sigma_1 - \Sigma_2}}{\sigma^2 - \normmop{\Sigma_1 - \Sigma_2}}  \right) \wedge \left( \frac{\normmop{\Sigma_1 - \Sigma_2}}{\sigma^2 - \normmop{\Sigma_1 - \Sigma_2}}\right)^2\right\}. \label{sstrength} 
\end{align}
Here, $\sigma^2 = \normmop{\Sigma_1}\vee \normmop{\Sigma_2}$ denotes the largest of the operator norms of the pre- and post-change covariance matrices. As before, the factor $\min(t_0, n-t_0)$ is the effective sample size, while the second factor measures the operator norm of the change relative to the difference between the nominal noise level and the magnitude of the change. As the second contribution of this paper, we show that the worst-case sum of Type I and Type II errors of any test can be made arbitrarily large whenever 
\begin{align}
    \rho \leq c \left\{ s \log \left( \frac{ep}{s}\right) \vee \log \log(8n)\right\}\label{skryt3},\end{align}
by choosing $c>0$ appropriately small. 
This is done in Section \ref{multivariatesec}. 

In a large region of the parameter space we match the minimax lower bound in \eqref{skryt3} from above with an optimal test statistic. Specifically, this is the low- to moderate-dimensional region where the effective sample size $\min(t_0,n-t_0)$ is no less than $C\left\{s\log(ep/s)\vee\log\log(8n)\right\}$, where the constant $C>0$ depends only on the desired testing level. In the high-dimensional region, where the effective sample size is even smaller, the lower bound still has interesting implications, even though the lower bound is not matched and thus not known to be tight. In this region, the lower bound in \eqref{skryt3} implies that the worst-case sum of Type I and Type II errors of any test can be made arbitrarily large whenever
\begin{align}
    \underset{v \in S^{p-1}_s}{\sup} \ \frac{v^\top \Sigma_1 v}{v^\top \Sigma_2 v} \vee \frac{v^\top \Sigma_2 v}{v^\top \Sigma_1 v} &\leq 1 + c \frac{  s \log \left( \frac{ep}{s}\right) \vee \log \log(8n)}{\min(t_0, n-t_0)},\label{implicationskryt}
\end{align} by choosing the constant $c>0$ appropriately small. The left-hand side of \eqref{implicationskryt} is the maximal ratio of the pre- and post-change variances along any subspace of $\RR^p$ spanned by an $s$-sparse vector. When $\min(t_0, n-t_0)$ is constant, the maximal ratio of the pre- and post-change variances must be of order $s \log(ep/s) \vee \log\log(8n)$ for changepoint detection to be possible, corresponding a stringent requirement on the signal strength of the changepoint. 

The optimal test statistic that we use to match the minimax lower bound in \eqref{skryt3} is based on $s$-sparse eigenvalues of empirical covariance matrices, which are known to be NP-hard to compute \citep{nphard}. As the third contribution of this paper, we propose a computationally feasible, near-optimal test for a change in covariance that is adaptive to the sparsity and noise level of the data. The test is constructed via a convex relaxation of the sparse eigenvalue problem using Semidefinite Programming, which has a computational cost that is polynomial in $p$. This is done in Section \ref{adaptivecompsec}.

Proofs of all mathematical results in the paper, along with various auxiliary lemmas, are given in the Supplementary Material. We use the following notation throughout the paper. For any set ${A}$, we let ${A}^\complement$ denote the complement of $A$ with respect to a universe $U \supseteq A$ determined from context. For any $d \in \NN$ we let $[d] = \{1,2, \ldots, d\}$. For any pair $x,y \in \RR$ we let $x \vee y = \max\{x,y\}$ and $x \wedge y = \min\{x,y\}$. Further, we let $\lfloor x \rfloor$ denote the largest integer no larger than $x$ and $\lceil x \rceil$ the smallest integer no smaller than $x$. We also write $x \lesssim y$ if $x \leq C y$ for some context-dependent absolute constant $C>0$, and we write $x \asymp y$ if $y \lesssim x$ holds as well. 
For any vector $v = (v(1), \ldots, v(d))^\top \in \RR^d$ and $p\geq 1$, we let $\normm{v}_p$ denote the $\ell_p$-norm of $v$ given by $\normm{v}_p = \{|v(1)|^p + \ldots + |v(d)|^p\}^{1/p}$. Moreover,  we let $\normm{v}_{\infty} = \max_{i \in [d]}  |v(i)|$ denote the $\ell_{\infty}$ norm of $v$, and we let $\normm{v}_0$ denote the $\ell_0$ "norm" of $v$, which is given by number of non-zero elements in $v$. For any matrix $A \in \RR^{d_1 \times d_2}$ and $p \in \{0\} \cup [1, \infty]$, we let $\normm{A}_p$ denote the $\ell_p$ norm of the vector formed by the entries of $A$. We define the operator norm of $A$ as $\normmop{A} = \sup_{v \in S^{d_2-1}} \normm{Ax}_2$, where $S^{d}$ denotes the Euclidean $d$-sphere. Moreover, if $A$ is symmetric, we define the largest absolute $s$-sparse eigenvalue of $A$ as $\lambda_{\max}^s(A) = \sup_{v \in S^{d_2-1}_s} |v^\top A v|$,  where $S^{d}_s = \{v \in S^{d} ; \normm{v}_0 \leq s\}$ denotes the $s$-sparse Euclidean unit sphere. We also write $A \succcurlyeq 0$ to mean that $A$ is positive semi-definite.
For any random variable $X$ taking values in $\RR$, we let $\normmo{X} = \inf \left\{ t>0: \ \EE \exp(X^2/t^2)\leq 2 \right\}$ denote the Orlicz-$\Psi_2$ norm of $X$, and we say that $X$ is sub-Gaussian if $\normmo{X}< \infty$.
For any $\RR^d$-valued random variable $X$, we let $\normmo{X} = \sup_{v \in S^{d-1}} \normmo{v^\top X}$ denote the Orlicz-$\Psi_2$ norm of $X$, saying that $X$ is sub-Gaussian if $\normmo{X}< \infty$. For any two probability measures $\mathbb{P}$ and $\mathbb{Q}$ on some measurable space $(\mathcal{X}, \mathcal{A})$, we define the total variation distance to be $\text{TV}(\mathbb{P}, \mathbb{Q}) = \sup_{A \in \mathcal{A}} \ \left | \mathbb{P}(A) - \mathbb{Q}(A)\right|$. If $\mathbb{P}$ is absolutely continuous with respect to $\mathbb{Q}$, we define the Chi-square divergence to be $\chi^2( \mathbb{P} \ \| \ \mathbb{Q}) = \int_{\mathcal{X}} ( \mathrm{d} \mathbb{P}^2 / \mathrm{d} \mathbb{Q}^2) \mathrm{d} \mathbb{Q} -1$. 
Throughout we let $\PP$ denote a generic probability measure determined from context, with associated expectation operator $\EE$. When the probability measure is specified, say to be $P$, we let the expectation operator $\EE_P$ and related functionals such as $\Cov_P(\cdot)$ be with respect to $P$.

\section{Results for univariate data}\label{univariatesec}
In this section we assume the observations $X_1, \ldots, X_n$ to have dimension $p=1$, in which case the covariance of $X_i$ is simply the variance.
We make the following assumption on the distribution of the $X_i$. 
\begin{assumption}\label{assunivariate} \phantom{linebreak} 
\begin{enumerate}[label={{\Alph*:}},
  ref={\theassumption.\Alph*}]
    \item \label{assunivariate-a}
        The $X_i$ are independent and mean-zero.
    \item \label{assunivariate-b}
        For some $w>0$ and all $i \in [n]$, the density of $X_i / (\EE X_i^2)^{1/2}$ is bounded above by $w$.
    \item \label{assunivariate-c}
        For some $u>0$ and all $i \in [n]$, it holds that $\normmo{X_i}^2 \leq u \EE X_i^2$.
\end{enumerate}
\end{assumption}

Note that Assumption \ref{assunivariate-b} guarantees that $X_i^2$ is bounded away from zero with high probability, needed for variance estimation, while Assumption \ref{assunivariate-c} ensures that the sub-Gaussian norm of the data is of the same order as the variance. 
 
The testing problem at hand is to determine whether the sequence of variances $\Var(X_i)$ is constant (the null hypothesis) or piece-wise constant with a single changepoint (the alternative hypothesis). Writing $X = (X_1, \ldots, X_n)^\top$, we formalize the testing problem by defining two sets of parameter spaces for $\text{Cov}(X)$, corresponding to each of the hypotheses. Fix any noise level $\sigma>0$. For the null hypothesis, we define
\begin{align}
    \Theta_0(n,\sigma) &= \left\{ \sigma^2 I_{n\times n} \right\},
\end{align}
which is a singleton set. For any $t_0 \in [n-1]$ and $\rho>0$, define
\begin{align}
    \Theta^{(t_0)}(n,\rho) &= \left\{ \vphantom{\frac{\normmop{\Sigma_1 - \Sigma_2}\sqrt{t_0}}{(\sigma^2)}}
                                \text{Diag} \left(  \left\{ V_i \right\}_{\in[n]}\right)  \ \in \RR^{n \times n} \ ; \ V_i = \sigma_1^2>0 \text{ for }  1\leq i \leq t_0,\right. \\
                                & \quad \ \ \ V_i = \sigma_{2}^2>0 \text{ for } t+1\leq i \leq n,\\
                                & \quad \ \ \ \left. \min(t_0, n - t_0)  \left[ \left( \frac{\left| \sigma_1^2 - \sigma_2^2\right| }{ \sigma_1^2 \wedge \sigma^2_2} \right) \wedge \left( \frac{\left| \sigma_1^2 - \sigma_2^2\right| }{ \sigma_1^2 \wedge \sigma^2_2}\right)^2\right] = \rho \right\},\label{altparamspaceuni}
\end{align}

The set $\Theta^{(t_0)}(n,\rho)$ is the space of covariance matrices of $X$ for which there is a change in variance at observation $t_0$ with signal strength $\rho$, and is non-empty for all combinations of $n \geq 2$, $t_0 \in [n-1]$ and $\rho>0$. The signal strength defined in \eqref{altparamspaceuni} depends on the ratio between the absolute change in variance and the smallest of the pre- and post-change variances. Interestingly, the signal strength depends quadratically on this ratio for weak signal strengths, and linearly for strong signal strengths. The signal strength is normalized by the effective sample size $\min(t_0, n-t_0)$ to ensure a common signal strength parameter across different changepoint locations. 
As our alternative hypothesis parameter space, we take
\begin{align}
    \Theta(n, \rho) = \bigcup_{t_0=1}^{n-1} \Theta^{(t_0)}(n,\rho).
\end{align}

We consider the problem of testing between $H_0: \Cov(X) \in \Theta_0(n,\sigma)$ and $H_1 : \Cov(X) \in \Theta(n,\rho)$. Let $\mathcal{P}(w,u)$ denote the class of distributions of $X = (X_1, \ldots, X_n)^\top$ for which the $X_i$ satisfy Assumption \ref{assunivariate} with $w,u>0$, and define
\begin{align}
    \mathcal{P}_0(n,w,u,\sigma) &= \{ P \in \mathcal{P}(w,u) \ ; \ \Cov_P(X) \in \Theta_0(n,\sigma) \}, \label{p0univariate}\\
   \mathcal{P}(n,w,u,\rho) &= \{ P \in \mathcal{P}(w,u) \ ; \ \Cov_P(X) \in \Theta(n,\rho)\}\label{punivariate},
\end{align}
i.e. the sub-classes of $\mathcal{P}(w,u)$ in accordance with the null hypothesis and the alternative hypothesis, respectively.  
Let $\Psi$ denote the class of measurable functions $\psi : \RR^{n} \mapsto \{0,1\}$. We define the minimax testing error $\mathcal{M}(\rho) = \mathcal{M}(\rho, n, \sigma,w,u)$ by 
\begin{align}
    \mathcal{M}(\rho)= \underset{\psi \in \Psi}{\inf} \left\{ \underset{P  \in \mathcal{P}_0(n,w,u,\sigma)}{\sup} \EE_{P} \psi(X) + \underset{P \in \mathcal{P}(n,w,u,\rho)}{\sup} \EE_{P} \left(1 - \psi(X)\right) \right\},
\end{align}
which measures the optimal "worst-case" performance of all testing procedures over the parameter space and the class of distributions satisfying Assumption \ref{assunivariate}. Our goal is to identify the minimax rate of testing, i.e. the boundary between feasible and unfeasible testing problems, defined as follows.
\begin{definition}
Fix any $w,u>0$. We say that $\rho^*(n) = \rho^*(n, \sigma, w,u)$ is the minimax rate of testing if the following conditions are satisfied:
\begin{enumerate}[label = {\Alph*:}]
\item For any $\delta \in (0,1)$, there exists some $c_{\delta}>0$ depending only on $\delta$, such that \\ $\mathcal{M}( c\rho^*(n), n, \sigma, w,u) \geq 1-\delta$ for any $c \in (0, c_{\delta})$.
\item For any $\delta \in (0,1)$, there exists some $C_{\delta}>0$ depending only on $\delta$, such that \\ $\mathcal{M}(C \rho^*(n), n, \sigma,w,u) \leq \delta$ for any $C \in (C_{\delta}, \infty)$.
\end{enumerate}
\end{definition}
The first main result of this paper is the following. 
\begin{theorem}\label{theorem1}
For any fixed $w\geq(2\pi)^{-1/2}$ and any fixed $u>0$ sufficiently large, the minimax rate of testing is given by 
\begin{align} 
\rho^*(n) \asymp \log\log(8n).
\end{align}
\end{theorem}
Note that the minimax rate of testing in Theorem \ref{theorem1} does not depend on the noise level $\sigma^2$ specified in the alternative hypothesis. In fact, we could also have taken $\cup_{\sigma>0} \Theta_0(\sigma,n)$ as the null hypothesis parameter space, obtaining the same result. We consider the restrictions on $w$ and $u$ to be artefacts of our proofs, as they are needed to ensure that Assumption \ref{assunivariate} is satisfied for Gaussian data, which is the distribution used to prove the lower bound in Theorem \ref{theorem1}. 

Theorem \ref{theorem1} shows an interesting similarity between the change-in-variance problem and the change-in-mean problem for univariate data. The signal strength of a single change in mean from $\mu_1$ to $\mu_2$ at observation $t_0$ with noise level $\sigma^2$ is given by $\min(t_0, n-t_0) (\mu_1 - \mu_2)^2 \sigma^{-2}$ \citep[see][]{liu_minimax_2021}, and the minimax rate of testing is given by $\log\log(8n)$ for this problem as well. 

However, it is important to realize that the above minimax result for the change-in-variance problem is not a consequence or corollary of existing results for mean changes. Indeed, the two problems are fundamentally different, as a change in variance alters the noise level of the data, unlike a change in the mean. In practice, this complicates the construction of a test statistic for a change in variance, as any variance estimate will typically have a variance depending on the very quantity it estimates. For instance, if $t_0$ were known, a natural estimate of the pre-change variance $\sigma_1^2$ would be $\widehat{\sigma}_1^2 = n^{-1}\sum_{i=1}^{t_0} X_i^2$, which has variance $2 \sigma_1^4 /t_0$ whenever the $X_i$ are independently Gaussian.

A noteworthy distinction between the change-in-variance and change-in-mean problems is the forms of the signal strengths. For the change-in-variance problem, the signal strength in  \eqref{altparamspaceuni} depends quadratically on a variance ratio for weak signals, and linearly for strong signals. Meanwhile, the signal strength for the change-in-mean problem does not exhibit such a phase transition. This discrepancy is intuitively explained by the sub-exponential tails of $X_i^2$ being heavier than the tails of $X_i$, the former only exhibiting sub-Gaussian behavior near the center of the distribution.

\subsection{Upper bound}\label{univariateuppersec}
To show Theorem \ref{theorem1}, we first consider the upper bound, for which an optimal test statistic is needed. Since we are only interested in testing for the presence of a changepoint, our approach is to apply location-specific tests over a sparse grid of candidate changepoint locations (as opposed to testing all candidate changepoint locations). If the changepoint location were known to be between $t$ and $n-t+1$, a natural approach would be to use the variance ratio statistic given by
\begin{align}
    S_t &= \frac{\widehat{\sigma}_{1,t}^2}{\widehat{\sigma}_{2,t}^2} \vee \frac{\widehat{\sigma}_{2,t}^2}{\widehat{\sigma}_{1,t}^2} -1,
\end{align}
where $\widehat{\sigma}_{1,t}^2$ and $\widehat{\sigma}_{t,t}^2$ are the empirical variances of $X_1, \ldots, X_t$ and $X_{n-t+1}, \ldots, X_n$, respectively, given by
\begin{align}
    \widehat{\sigma}^2_{1,t} &= \frac{1}{t}\sum_{i=1}^t X_i^2, & \widehat{\sigma}^2_{2,t} = \frac{1}{t} \sum_{i=1}^t X_{n-i+1}^2.
\end{align}
Under the null hypothesis, $S_t$ is a ratio of two sub-Exponential random variables with  Orlicz-1 norms of the same order. Since also $X_i / \{\EE X_i^2\}^{1/2}$ has a bounded density for each $i$, both the lower and upper tails of $\widehat{\sigma}^2_{1,t}$ and $\widehat{\sigma}^2_{2,t}$ can be well controlled, allowing for high-probability control on the tails of $S_t$. Under the alternative hypothesis, $S_t$ grows linearly with the signal strength, since both $\widehat{\sigma}^2_{1,t}$ and $\widehat{\sigma}^2_{2,t}$ are unbiased estimators of the pre- and post-change variances, respectively. A natural testing procedure is therefore to reject the null hypothesis for large values of $S_t$. 

However, the true changepoint location $t_0$ under the alternative is unknown, which motivates the application of the aforementioned testing procedure across a range of values of $t$. To attain the level of precision needed to reach the minimax testing rate, we use the geometrically increasing grid in \cite{liu_minimax_2021}, given by
\begin{align}
    \mathcal{T} = \left\{ 2^0, 2^1, \ldots, 2^{\lfloor \log_2(n/2)\rfloor}\right\}\label{mathcalt},
\end{align}
The main advantage of $\mathcal{T}$ is its small size, having a cardinality of order $\log n$. Still, due to the spacing of $\mathcal{T}$, the signal strength of the data is preserved for at least one $t \in \mathcal{T}$, regardless of the true changepoint location $t_0$. Given a tuning parameter $\lambda>0$, our testing procedure for a change in variance is given by
\begin{align}
    \psi(X) = \psi_{\lambda}(X) = \max_{t \in \mathcal{T}} \ind\left\{       S_t > \lambda \left( \sqrt{\frac{\log\log(8n)}{t}} \vee  \frac{\log\log(8n)}{t}\right)\right\}.\label{univariatetest}
\end{align}

The following theorem gives the theoretical performance of the testing procedure defined in \eqref{univariatetest}. 

\begin{proposition}\label{univariateupper}
    Fix any $\delta \in(0,1)$, $\sigma>0$, $w>0$ and $u>0$. Let $\mathcal{P}_0(n,\sigma,w,u)$ and $\mathcal{P}(n,\rho,w,u)$ be as in \eqref{p0univariate} and \eqref{punivariate}, respectively. Then for some constant $\lambda_0>0$ depending only on $\delta, w$ and $u$, the testing procedure in \eqref{univariatetest} satisfies 
    \begin{align}
       \underset{P  \in \mathcal{P}_0(n,\sigma,w,u)}{\sup} \EE_{P} \psi(X) + \underset{P \in \mathcal{P}(n,\rho,w,u)}{\sup} \EE_{P} \left(1 - \psi(X)\right)  \leq \delta,
    \end{align}
    whenever $\lambda \geq \lambda_0$ and $\rho \geq C \log\log(8n)$, where $C$ is a constant depending only on $\lambda$, $\delta$,  $w$ and $u$.
\end{proposition}
Some remarks are in order. Firstly, the explicit values of constants in Proposition \ref{univariateupper} can be found in the proof, although we note that these have not been optimized. Secondly, we remark that these constants do not depend on $\sigma$, meaning that the testing procedure in \eqref{univariatetest} is adaptive to the nominal noise level of the data. The test can therefore be used to test the composite null hypothesis $H_0 \ : V \in \cup_{\sigma>0} \Theta_0(n, \sigma)$ versus the alternative $H_1$ without altering the tuning parameter $\lambda$. Lastly, we remark that the testing procedure in \eqref{univariatetest} is easy to implement and has computational cost. Indeed, the test is simply a ratio of empirical variances evaluated over a grid. Since the empirical variances can be computed from cumulative sums of the squared $X_i$, the computational cost of the testing procedure is of order $n$.

\subsection{Lower bound}
The following lower bound matches the upper bound on $\rho^*(n)$ implied by Proposition \ref{univariateupper}. 
\begin{proposition}\label{univariatelower}
    For any fixed $\delta \in (0,1)$, $\sigma>0$, $w \geq (2\pi)^{-1/2}$ and sufficiently large $u>0$, there exists some $c_{\delta}>0$ depending only on $\delta$, such that $\mathcal{M}(\rho) \geq 1-\delta$ whenever $\rho \leq c \log \log(8n)$ and $c \in (0, c_{\delta})$.
\end{proposition}

The proof of Proposition \ref{univariatelower} adopts techniques from \cite{liu_minimax_2021}, and the proposition is shown by bounding the Chi square divergence between two probability measures $\PP_0$ and $\PP$ on $\RR^{n}$, where $\PP_0$ is consistent with the null hypothesis and $\PP$ is consistent with the alternative. Specifically, $\PP_0$ is the distribution of $X = (X_1,\ldots, X_n)^\top$ when $X_i~\sim~\text{N}(0, \sigma^2)$ independently for $i=1, \ldots, n$, and $\PP$ is the distribution of $X$ induced by sampling the changepoint location $t_0$ uniformly from $\{2^0,2^1, \ldots, 2^{\lfloor \log_2(n/2)\rfloor}\}$, and conditionally on $t_0$, sampling $X_i \sim\text{N}(0, \sigma^2 - \kappa)$ independently for $i \leq t_0$ and $X_i \sim \text{N}(0, \sigma^2)$ independently for $i>t_0$. The scaling factor $\kappa < \sigma^2$ is chosen so that $\text{Cov}(X | t_0) \in \Theta(n, \rho)$ for all $t_0$.

\section{Results for multivariate data}\label{multivariatesec}
In this section we turn to multivariate data, assuming now that the data $X_1, \ldots, X_n$ are $p$-dimensional with $p\geq 1$. 
Similar to the univariate case, we impose the following assumption on the distribution of the $X_i$.
\begin{assumption}\label{assmultivariate} \phantom{linebreak} 

\begin{enumerate}[label={{\Alph*:}},
  ref={\theassumption.\Alph*}]
    \item \label{assmultivariate-a}
        The $X_i$ are independent and mean-zero.
    \item \label{assmultivariate-b}
        For some $u>0$ and all $i \in [n]$, it holds that $\normmo{X_i}^2 \leq u \lVert {\EE X_i X_i^\top}\rVert_{\mathrm{op}}$.
\end{enumerate}
\end{assumption}

The testing problem at hand is to determine whether the sequence of covariances $\EE X_i X_i^\top$ is constant (the null hypothesis), or piece-wise constant with a single change point (the alternative hypothesis). Write $X = (X_1^\top, \ldots, X_n^\top)^\top$ shorthand for the stacked vector consisting of all the $X_i$. We define two sets of parameter spaces for $\Cov(X)$, corresponding to each of the hypotheses. Let $\PD(p)$ denote the positive definite cone, consisting of all symmetric and positive definite $p \times p$ matrices. As the parameter space for the null hypothesis, for any $\sigma>0$ we define
\begin{align}
    \Theta_0(p,n,\sigma) &= \left\{  \vphantom{\normmop{\Sigma} \leq \sigma^2}  \text{Diag} \left( \left\{ \Sigma \right\}_{i\in[n]}\right)  \ \in \RR^{pn \times pn} \ ; \ \Sigma \in \PD(p) \ , \ \normmop{\Sigma} = \sigma^2 \right\}.\
\end{align}

Recall that we say a symmetric matrix is $s$-sparse if its operator norm agrees with its $s$-sparse largest absolute eigenvalue, i.e. 
\begin{align}
    \underset{v \in S^{p-1}_s}{\sup} \ |v^\top (\Sigma_1 - \Sigma_2)v | &= \underset{v \in S^{p-1}}{\sup} \ |v^\top (\Sigma_1 - \Sigma_2)v |.
\end{align}
Now, let $M(p, s)$ denote the space of $s$-sparse matrices in $\RR^{p\times p}$. For any $t_0 \in [n-1]$,  $s \in [p]$ and $\rho>0$, we define $\Theta^{(t_0)}(p,n, s,\sigma,\rho)$ to be the set given by
\begin{align}
\left\{ \vphantom{\frac{\normmop{\Sigma_1 - \Sigma_2}\sqrt{t_0}}{(\sigma^2)}} \right.  & \text{Diag} \left(  \left\{ V_i \right\}_{i\in [n]}\right)  \ \in \RR^{pn \times pn}, \  V_i = \Sigma_1 \in \PD(p) \text{ for } 1\leq i \leq t_0, \\ 
                                & V_i = \Sigma_{2} \in \PD(p)  \text{ for } t_0+1\leq i \leq n,\  \normmop{\Sigma_1} \vee \normmop{\Sigma_2} \leq \sigma^2,\\
                                 &  \Sigma_1 - \Sigma_2 \in M(p,s),\\
                                & \left. \min(t_0, n - t_0) \left[  \left( \frac{\normmop{\Sigma_1 - \Sigma_2}}{\sigma^2 - \normmop{\Sigma_1 - \Sigma_2}}  \right) \wedge \left( \frac{\normmop{\Sigma_1 - \Sigma_2}}{\sigma^2 - \normmop{\Sigma_1 - \Sigma_2}}\right)^2\right] = \rho \right\}. \ \ \  \label{altparammulti}
\end{align}
The set $\Theta^{(t_0)}(p,n, s, \sigma,\rho)$ contains the space of covariance matrices of $X$ with signal strength $\rho$ and nominal noise level at most $\sigma^2$, for which there is an $s$-sparse change in covariance at time $t_0$, and is non-empty for all combinations of $n\geq 2$, $p\geq 1$, $s \in [p]$,  $t_0 \in [n-1]$, $\sigma>0$, and $\rho>0$. Here, $n$ and $p$ determine the sample size and dimension of the problem, while $\rho$ and $s$ determine the signal strength and the sparsity level, respectively. 
The signal strength is given by the operator norm of the covariance change, normalized by the inverse difference between the noise level of the data and the norm of the covariance change. As in the univariate setting, the signal strength is normalized by the effective sample size $\min(t_0, n-t_0)$, ensuring a common signal strength parameter across different changepoint locations. 
As our alternative hypothesis parameter space, we take
\begin{align}
    \Theta(p,n,s, \sigma, \rho) = \bigcup_{t_0=1}^{n-1} \Theta^{(t_0)}(p,n, s,\sigma, \rho).
\end{align}

We consider the problem of testing between $H_0: \Cov(X) \in \Theta_0(p,n,\sigma)$ and $H_1 : \Cov(X) \in \Theta(p,n,s,\sigma, \rho)$. Let $\mathcal{P}(u)$ denote the set of distributions $X = (X_1^\top, \ldots, X_n^\top)^\top$ for which the $X_i$ satisfy Assumption \ref{assmultivariate} with $u>0$, and define
\begin{align}
    \mathcal{P}_0(p,n,u,\sigma) &= \{ P \in \mathcal{P}(u) \ ; \ \Cov_P(X) \in \Theta_0(p,n,\sigma) \}, \label{p0multivariate}\\
    \mathcal{P}(p,n,s,u,\sigma,\rho) &= \{ P \in \mathcal{P}(u) \ ; \ \Cov_P(X) \in \Theta(p,n,s,\sigma,\rho)\}\label{pmultivariate},
\end{align}
i.e. the sub-classes of $\mathcal{P}(u)$ in accordance with the null hypothesis and the alternative hypothesis, respectively. 
Let $\Psi$ denote the class of measurable functions $\psi : \RR^{pn} \mapsto \{0,1\}$. We define the minimax testing error $\mathcal{M}(\rho) = \mathcal{M}(p,n,s,u,\sigma,\rho)$ by 
\begin{align}
   \mathcal{M}(\rho) = \underset{\psi \in \Psi}{\inf} \left\{ \underset{P \in \mathcal{P}_0(p,n,u,\sigma)}{\sup} \EE_{P} \psi(X) + \underset{P \in \mathcal{P}(p,n,s,u,\sigma,\rho)}{\sup} \EE_{P} \left(1 - \psi(X)\right) \right\}.
\end{align}

\subsection{A lower bound on the minimax testing error}\label{multivariatelowersec}
The following result gives a lower bound on the minimax testing error. 
\begin{proposition}\label{multivariatelowerbound}
    Fix any $\delta \in (0,1)$, $\sigma>0$, $s \in [p]$ and sufficiently large $u>0$. Then there exists some $c_{\delta}>0$ depending only on $\delta$, such that $\mathcal{M}(\rho) \geq 1-\delta$ whenever \\$\rho \leq c \left\{s \log(ep/s) \vee \log \log(8n) \right\}$ and $c \in (0, c_{\delta})$.
\end{proposition}
Proposition \ref{multivariatelowerbound} implies that no changepoint procedure can discriminate between $H_0$ and $H_1$ with a worst-case testing error less than any fixed $\delta\in (0,1)$ whenever the signal strength defined in \eqref{altparammulti} is smaller than $s \log(ep/s) \vee \log \log(8n)$ by a sufficiently small factor. This lower bound is increasing sub-linearly in the sparsity $s$, reducing to $\log(ep)\vee \log\log(8n)$ when $s=1$ and growing to as much as $p \vee \log\log(8n)$ when $s=p$. Note that the constant in Proposition \ref{multivariatelowerbound} does not depend on the noise level $\sigma$. In fact, we could also have taken $\cup_{\sigma>0} \Theta_0(p,n,\sigma)$ as the null hypothesis parameter space and $\cup_{\sigma>0} \Theta(p,n,s,\sigma,\rho)$ as the alternative, obtaining the same result.

In the sequel, we will match the minimax lower bound in Proposition \ref{multivariatelowerbound} from above for all changepoints whose location $t_0$ satisfies $\min(t_0, n-t_0) \gtrsim s \log\left(\frac{ep}{s}\right) \vee \log \log(8n)$, which we call the low- to moderate-dimensional region of the parameter space. For high-dimensional problems where $\min(t_0, n-t_0) \lesssim \log\left(\frac{ep}{s}\right) \vee \log \log(8n)$, Proposition \ref{multivariatelowerbound} has interesting implications, even though the lower bound there is not matched from above. Indeed, due to Lemma \ref{boundimplication} in the Supplementary Material, Proposition \ref{multivariatelowerbound} implies that for any test to discriminate between $H_0$ and $H_1$ with minimax testing error smaller than some $\delta \in (0,1)$, a necessary condition is that
\begin{align}
    \underset{v \in S^{p-1}_s}{\sup} \ \frac{v^\top \Sigma_1 v}{v^\top \Sigma_2 v} \vee \frac{v^\top \Sigma_2 v}{v^\top \Sigma_1 v} &\geq 1 + c \frac{ \gamma}{\Delta} \vee \sqrt{c \frac{ \gamma}{\Delta}},\label{variancedirection}
\end{align}
for some absolute constant $c>0$ depending only on $\delta$, where $\gamma = \gamma(p,n,s) = s \log\left(\frac{ep}{s}\right) \vee \log \log(8n)$ and $\Delta = \Delta(n,t_0) = \min(t_0, n-t_0)$ is the effective sample size. The left-hand side of \eqref{variancedirection} is the largest relative change in the variance of the data along some subspace of $\RR^p$ spanned by an $s$-sparse vector. In the asymptotic regime where $n,s$ or $p$ diverge, the right-hand side of \eqref{variancedirection} diverges if (and only if) $\Delta / \gamma \rightarrow 0$. If $\Delta$ is fixed, the right-hand side of \eqref{variancedirection} grows at the rate of $\gamma = s \log\left(\frac{ep}{s}\right) \vee \log \log(8n)$. In this case, the relative change in variance along some axis of variation must change by a factor of order $\gamma$ for changepoint detection to be possible, amounting to a very stringent requirement, especially when $s$ or $p$ become large. To see this, consider for instance the case where $s=p$, in which the necessary relative change must be of order at least $p$ for any test to successfully discriminate between $H_0$ and $H_1$ over the whole parameter space.

The proof of Proposition \ref{multivariatelowerbound} is similar to that of Proposition \ref{univariatelower}. The main strategy is to bound the Chi square divergence between two probability measures $\PP_0$ and $\PP$ on $\RR^{pn}$, where $\PP_0$ is consistent with the null hypothesis and $\PP$ is consistent with the alternative. Specifically, $\PP_0$ is the distribution of $X= (X_1^\top,\ldots, X_n^\top)^\top$ when $X_i \sim \text{N}_p(0, \sigma^2I)$ independently for all $i$. Moreover, $\PP$ is the mixture distribution of $(X_1, \ldots, X_n)$, which conditional on the changepoint location $t_0$ and the pre-change covariance matrix $\Sigma$ satisfies $X_i \sim \text{N}_p(0,\Sigma)$ independently for $i \leq t_0$ and $X_i \sim \text{N}_p(0, \sigma^2I)$ independently for $i \geq t_0+1$. To generate $t_0$ and $\Sigma$, we sample $t_0$ uniformly from $\{2^0,2^1, \ldots, 2^{\lfloor \log_2(n/2)\rfloor}\}$, and independently sample $\Sigma = \sigma^2 I - \kappa u u^\top$, where $\kappa <\sigma^2$ is deterministic chosen to satisfy the signal strength condition, and $u$ is suitably sampled from the $s$-sparse unit sphere $S^{p-1}_s$. Here, the stochastic choice of $t_0$ contributes with the $\log \log(8n)$ term in the lower bound, while the stochastic choice of $u$ contributes with the $s \log(ep/s)$ term. 

Let us compare Proposition \ref{multivariatelowerbound} to the minimax result in \cite{covariancecusum}. Their result concerns changepoint location estimation, which is a slightly different from the testing problem considered in this paper. \citeauthor{covariancecusum} assume that $\normmop{\Sigma_1 - \Sigma_2}\leq \sigma^2/4$, where $\sigma^2$ denotes the maximum sub-Gaussian norm of the $X_i$. The latter assumption implicitly restricts the effective sample size to satisfy $\min(t_0, n-t_0)\gtrsim p$. Within this setup, they show that consistent estimation of a changepoint location is impossible as long as
\begin{align}
\min(t_0, n-t_0) \frac{\normmop{\Sigma_1 - \Sigma_2}^2}{\sigma^4} \lesssim p.\label{wangbound} 
\end{align}
Under the assumption that $\normmop{\Sigma_1 - \Sigma_2}\leq \sigma^2/4$, the left-hand side of \eqref{wangbound} agrees with the signal strength defined in \eqref{altparammulti}, up to absolute constants. Thus, Proposition \ref{multivariatelowerbound} significantly refines the minimax result in \cite{covariancecusum} by accounting for the dependence on both the sparsity and the sample size. 

Next we compare the lower bound in Proposition \ref{multivariatelowerbound} to the minimax testing rate for changes in the mean vector. Assuming independent data, consider testing the null hypothesis 
\begin{align}
    H_{0,\text{mean}} \ : \ X_i \sim \text{N}_p(\mu, \sigma^2I), \text{ for } i \in [n],\\
    \intertext{where $\mu \in \RR^p$ is unknown, versus the alternative $H_{1,\text{mean}}$ that}
    X_i\sim \begin{cases} \text{N}_p(\mu_1, \sigma^2 I) & \text{for } 1 \leq i \leq t_0,\\
    \text{N}_p(\mu_2, \sigma^2 I) & \text{for } t_0+1 \leq i \leq n,
    \end{cases}
\end{align}
where $\normm{\mu_1 - \mu_2}_0 = s$ is fixed and $\mu_1, \mu_2\in \RR^p$ and $t_0\in [n-1]$ are unknown. Given any $\delta \in (0,1)$, \cite{liu_minimax_2021} show that no test can discriminate between $H_{0,\text{mean}}$ and $H_{1,\text{mean}}$ with minimax testing error less than $\delta$ whenever
\begin{align}
    \Delta \frac{\normm{\mu_1 - \mu_2}^2}{\sigma^2} \lesssim \begin{cases} \sqrt{p \log\log(8n)}, &\text{ if } s \geq \sqrt{p\log\log(8n)},\\
    s \log\left(\frac{ep\log\log(8n)}{s^2}\right) \vee \log\log(8n), &\text{ if } s < \sqrt{p\log\log(8n)}, \end{cases} \ \ \ \ \label{meanminimax}
\end{align}
where $\Delta = \min(t_0, n-t_0)$.
Within the mean-change model, the regime $s < \sqrt{p\log\log(8n)}$ is called a sparse regime, and $s \geq \sqrt{p\log\log(8n)}$ is called a dense regime. In comparison, Proposition \ref{multivariatelowerbound} implies that no test can discriminate between $H_0$ and $H_1$ with minimax testing error less than $\delta$ whenever
\begin{align}
& \Delta \left\{ \left( \frac{\normmop{\Sigma_1 - \Sigma_2}}{\sigma^2 - \normmop{\Sigma_1 - \Sigma_2}}  \right) \wedge \left( \frac{\normmop{\Sigma_1 - \Sigma_2}}{\sigma^2 - \normmop{\Sigma_1 - \Sigma_2}}\right)^2\right\} \\
\lesssim & s \log\left(\frac{ep}{s}\right) \vee \log \log(8n),\label{lowerboundimplication}
\end{align}
where $\Delta$ is defined as above. 
Interestingly, the minimax testing rate for a change in the mean features a phase transition between the dense and sparse regimes, resulting in an elbow effect in the minimax rate in \eqref{meanminimax}. Such a phase transition is not present for the covariance changepoint problem, in which no phase transition occurs with respect to the sparsity $s$. A different phase transition occurs, however, in the effect of the variance ratio in the left-hand side of \eqref{lowerboundimplication}, being linear for large values of the variance ratio and quadratic for smaller values.

\subsection{Upper bound on the minimax testing error}\label{multivariateuppersec}
In this section we present a test statistic that attains the minimax testing rate for low- to moderate-dimensional problems, thereby matching the minimax lower bound from Section \ref{multivariatelowersec} in this region of the parameter space. To keep the presentation simple, we assume both the variance parameter $\sigma^2$ and the sparsity $s$ under the alternative to be known. Note that these assumptions can be removed, which is done in the next subsection. As in Section \ref{univariateuppersec}, our approach is to apply location-specific tests over a grid of candidate changepoint locations. To detect changepoints between $t$ and $n-t+1$, we construct the following statistic: 
\begin{align}
    S_{t,s} &= \lambda_{\max}^s (\widehat{\Sigma}_{1,t} - \widehat{\Sigma}_{2,t}),\label{stsdef}
    \end{align}
    where
    \begin{align}
    \widehat{\Sigma}_{1,t} &= \sum_{i=1}^t X_i X_i^\top, &\widehat{\Sigma}_{2,t} &= \sum_{i=1}^t X_{n-i+1} X_{n-i+1}^\top\label{sigmahats}.
\end{align}

The statistic $S_{t,s}$ measures the absolute difference between the empirical covariance matrices of $X_1, \ldots, X_t$ and $X_{n-t+1}, \ldots, X_n$ in terms of the largest absolute $s$-sparse eigenvalue. When $s=p$, the statistic $S_{t,s}$ is reminiscent of the Covariance CUSUM statistic defined in \cite{covariancecusum}.  Due to the sub-Gaussianity of the $X_i$, high-probability control over the tails of $S_{t,s}$ can be obtained under the null hypothesis using standard concentration inequalities for empirical covariance matrices. Meanwhile, under the alternative hypothesis, $S_{t,s}$ grows linearly with the signal strength, as both $\widehat{\Sigma}_{1,t}$ and $\widehat{\Sigma}_{2,t}$ are unbiased estimators of the pre- and post-change covariances, respectively. A natural testing procedure is therefore to reject the null hypothesis for large values of $S_{t,s}$. To guarantee statistical power over all possible changepoint locations, we apply $S_{t,s}$ over the grid $\mathcal{T}$ given in \eqref{mathcalt}. 

Given a tuning parameter $\lambda>0$, our testing procedure for a change in covariance is given by 
\begin{align}
    \psi(X) &= \psi_{\lambda}(X) = \max_{t \in \mathcal{T}} \ind\left\{       S_{t,s} > \lambda \sigma^2 r(p,n,s,t)\right\},\label{multivariatetest}
    \end{align}
where
\begin{align}
    r(p,n,s,t) &=  \sqrt{\frac{s \log(ep/s) \vee \log\log(8n)}{t}} \vee  \frac{s \log (ep/s) \vee \log\log(8n)}{t}.\label{rdef}
\end{align}

The following theorem gives the theoretical performance of the test in \eqref{multivariatetest}. 
\begin{proposition}\label{multivariateupperbound}
    Fix any $\delta \in (0,1)$, $\sigma>0$, $s \in [p]$ and $u>0$. Let $P_0 \in \mathcal{P}_0(p,n,u,\sigma)$ and $P \in \mathcal{P}(p,n,s,u,\sigma,\rho)$, where $ \rho \geq C \left\{ s \log(ep/s) \vee \log\log(8n)\right\}$ for some $C>0$ and the sets $ \mathcal{P}_0(p,n,u,\sigma)$ and $\mathcal{P}(p,n,s,u,\sigma,\rho)$ are defined in \eqref{p0multivariate} and \eqref{pmultivariate}, respectively. Then there exists some $\lambda_0>0$ depending only on $\delta$ and $u$, such that the testing procedure in \eqref{multivariatetest} with $\lambda \geq \lambda_0$ satisfies
    \begin{align}
         \EE_{P_0} \psi(X) +  \EE_{P} \left(1 - \psi(X)\right) \leq \delta, 
    \end{align}
    as long as $C \geq 8 \lambda^2$ and the distribution $P$ under the alternative additionally satisfies
    \begin{align}
        \min(t_0,n-t_0) \frac{\normmop{\Sigma_1 - \Sigma_2}^2}{\sigma^4}  \geq  C \left\{ s \log(ep/s) \vee \log\log(8n)\right\} \label{prop4condition},
    \end{align}
    where $t_0, \Sigma_1$ and $\Sigma_2$ respectively denote the changepoint location and pre- and post-change covariance matrices of the $X_i$ implied by $P$. 
\end{proposition}

Up to absolute constants, Proposition \ref{multivariateupperbound} matches the minimax lower bound in Proposition \ref{multivariatelowerbound} in the low- to moderate-dimensional region of the alternative parameter space where $\min(t_0, n-t_0) \geq  C \left\{ s \log(ep/s) \vee \log\log(8n)\right\}$. Here, $C$ is the constant from Proposition \ref{multivariateupperbound} depending only on the desired testing level $\delta \in (0,1)$ and the ratio between the sub-Gaussian norm and the variance of the $X_i$. Indeed, in this region, Proposition \ref{multivariatelowerbound} implies that the minimax testing error is at least $1-\delta$ whenever
\begin{align}
    \min(t_0, n-t_0) \frac{\normmop{\Sigma_1 - \Sigma_2}^2}{\sigma^4} &\leq c \left\{ s\log(ep/s)\vee \log\log(8n)\right\},
\end{align}
and the constant $c>0$ is sufficiently small.

\subsection{Adaptivity to noise level and sparsity}\label{adaptivesec}
The testing procedure in \eqref{multivariatetest} is minimax rate optimal in a large region of the parameter space, but requires knowledge of the sparsity $s$ and variance parameter $\sigma^2$. In this section we present a modified test statistic that is adaptive to these quantities. Since the test in \eqref{multivariatetest} is only guaranteed to detect an $s$-sparse change in covariance when the effective sample size is of the same order as $\gamma(s) = \gamma(p,n,s) = s \log(ep/s)\vee \log\log(8n)$, a natural estimate of the (scaled) variance of $S_{t,s}$ is given by 
\begin{align}
    \widehat{\sigma}^2_s = \lambda_{\max}^s (\widehat{\Sigma}_{1,\lceil \gamma(s) \rceil}) \wedge \lambda_{\max}^s (\widehat{\Sigma}_{2,\lceil \gamma(s) \rceil}) \label{sigmahatsdef} ,
\end{align}
where $\widehat{\Sigma}_{i,t}$ is defined in \eqref{sigmahats}, for $i=1,2$ and $t \in [n]$. In \eqref{sigmahatsdef}, the first and last $\lceil \gamma(p,n,s)\rceil$ data points contribute to the variance estimate, as no changepoints in these segments can be guaranteed to be detected by our testing procedure anyway. It is therefore in \eqref{sigmahatsdef} implicitly assumed that $n \geq  \gamma(s)$. Given a tuning parameter $\lambda >0$, our adaptive variant of the test in \eqref{multivariatetest} is defined by 
\begin{align} 
    \psi_{\mathrm{adaptive}}(X) &= \psi_{\mathrm{adaptive},\lambda}(X) = \underset{t \in \mathcal{T}}{\max} \ \underset{\substack{s \in \mathcal{S}; \\ \gamma(s) \leq n}}{\max} \ \ind\left\{       S_{t,s} > \lambda \widehat{\sigma}^2_s r(p,n,s,t)\right\}, \ \ \label{adaptivemultivariatetest}
    \end{align}
    where $S_{t,s}$ is given in \eqref{stsdef},  $r(p,n,s,t)$ is given in \eqref{rdef}, $\mathcal{T}$ is given in \eqref{mathcalt} and 
    \begin{align}
    \mathcal{S} &= \left\{ 2^0, 2^1, \ldots, 2^{\lfloor \log_2(p)\rfloor}\right\}\label{mathcals}
\end{align}
is a geometric grid of candidate sparsities. Note that this geometric grid is sufficient to retain power for all possible sparsities $s \in [p]$. Unlike the test in \eqref{multivariatetest},  the test in \eqref{adaptivemultivariatetest} is adaptive to the noise level by using an estimate for the noise level, and is adaptive to the sparsity level by testing over all sparsities in the grid $\mathcal{S}$. 

For the theoretical analysis of the testing procedure in \eqref{adaptivemultivariatetest}, we impose the following assumption, which is slightly stronger than Assumption \ref{assmultivariate}:
\begin{assumption}\label{assmultivariate2} \phantom{linebreak} 

\begin{enumerate}[label={{\Alph*:}},
  ref={\theassumption.\Alph*}]
    \item \label{assmultivariate2-a}
        The $X_i$ are independent and mean-zero.
    \item \label{assmultivariate2-b}
        For some $w>0$, all $i \in [n]$ and all $v \in S^{p-1}$, the random variable \\$v^\top X_i /  \{ \EE (v^\top X_i X_i^\top v)\}^{1/2}$ has a continuous density bounded above by $w$.
    \item \label{assmultivariate2-c}
        For some $u>0$, all $i \in [n]$, and all $v \in S^{p-1}$, we have\\$ \lVert v^\top X_i \rVert_{\Psi_2}^2 \leq u \EE  \{ (v^\top X_i)^2 \}$.
\end{enumerate}
\end{assumption}
Assumption \ref{assmultivariate2-b} ensures that $\widehat{\sigma}^2_s$ is bounded away from zero with high probability, while Assumption \ref{assmultivariate2-c} ensures that the sub-Gaussian norm of the data, along any axis of variation, is of the same order as the variance. In particular, this allows testing procedure to adapt to models in which $\lambda_{\max}^s(\EE X_i X_i^\top)$ is of smaller order than $\lVert{\EE X_i X_i^\top}\rVert_{\mathrm{op}}$. The following theorem gives the theoretical performance of $ \psi_{\mathrm{adaptive}}$. 
\begin{proposition}\label{adaptivemultivariateupperbound}

Assume that $X_1, \ldots, X_n$ satisfy Assumption \ref{assmultivariate2} for some $w,u>0$. Let $\EE_0$ denote the expectation operator with respect to the distribution of $X = (X_1^\top, \ldots, X_n^\top)^\top$ when $\Cov(X_i)$ is constant and positive definite, and let $\EE_1$ denote the expectation operator when $\Cov(X_i) = \Sigma_1$ for $i \leq t_0$ and $\Cov(X_i) = \Sigma_2$ for $i> t_0$, where $\Sigma_1$ and $\Sigma_2$ are some positive definite  
matrices.

    Given any $\delta \in (0,1)$,  there exists some $\lambda_0>0$ depending only on $\delta$, $w$ and $u$, such that the test in \eqref{adaptivemultivariatetest} with $\lambda \geq \lambda_0$ satisfies
    \begin{align}
         \EE_{0} \psi_{\mathrm{adaptive}}(X) +  \EE_{1} \left(1 - \psi_{\mathrm{adaptive}}(X)\right) \leq \delta
    \end{align}
    as long as
    \begin{align}
        \min(t_0,n-t_0)\left( \frac{ \lambda_{\max}^s(\Sigma_1 - \Sigma_2)}{\lambda_{\max}^s(\Sigma_1) \vee \lambda_{\max}^s(\Sigma_2)} \right)^2 \geq  C \left\{ s \log(ep/s) \vee \log\log(8n)\right\}, \ \ \ \label{prop5condition}
    \end{align}
    for some $s \in [p]$, where $C>0$ is some positive constant depending only on $\lambda, \delta, w$ and $u$. 
\end{proposition}
Note that the theoretical performance of the adaptive testing procedure in \eqref{adaptivemultivariatetest} is stronger than that of the testing procedure in \eqref{multivariatetest}, since the condition in \eqref{prop5condition} is weaker than the condition in \eqref{prop4condition}. This is due to the adaptivity of \eqref{adaptivemultivariatetest} to both the sparsity and noise level of the data.

\section{An adaptive and computationally tractable multivariate changepoint procedure} \label{adaptivecompsec}

\subsection{A convex relaxation of the k-sparse eigenvalue problem}
The testing procedures presented in the previous section have provably strong theoretical performance, but are computationally intractable for all but small values of $p$. Indeed, a key ingredient in these tests is the $s$-sparse eigenvalue of the $p\times p$ matrix  $\widehat{\Sigma}_{1,t} - \widehat{\Sigma}_{2,t}$. Seeing as the sparse eigenvalue is NP hard to compute in general \citep{nphard}, the changepoint testing procedures in the previous section are thus prohibitively computationally costly unless $p$ is very small. 

As a remedy to the large computational cost, the $s$-sparse eigenvalue can be approximated by a convex relaxation of the implicit optimization problem via Semidefinite Programming, as is done in \cite{berthet_optimal} for testing for a rank one perturbation in an isotropic covariance matrix.  Following \cite{berthet_optimal} and \cite{convex_relaxation}, we can reformulate the sparse eigenvalue problem as follows. For any symmetric matrix $A \in \RR^{p\times p}$, recalling that 
\begin{align}
    \lambda_{\max}^s(A) &= \underset{v \in S^{p-1}_s}{\sup} \ |v^\top A v|,\label{sparseeigreiteration} 
\end{align}
we can for any $v \in S^{p-1}_s$ make a change of variables $Z = v v^\top$. Noticing that $v^\top A v =  \Tr(A v v^\top) = \Tr(AZ)$,  $\mathrm{Tr}(Z) = 1$, $\mathrm{rank}(Z)=1$,  $\normm{Z}_0 \leq s^2$ and $Z \succcurlyeq0$, the optimization problem in \eqref{sparseeigreiteration} can be rewritten as
\begin{align}
    \lambda_{\max}^s(A) = \underset{Z \in M(p,s)}{\sup} \  | \mathrm{Tr}(AZ) |,\label{sparseeigalt}
\end{align}
where
\begin{align}
    M(p,s) = \left\{ Z \in \RR^{p\times p} \ ; \ Z \succcurlyeq 0, \rank(Z) = 1, \Tr(Z) = 1, \normm{Z}_0 \leq s^2 \right\}.
\end{align}
The optimization problem in \eqref{sparseeigalt} features two sources of non-convexity; the $\ell_0$ constraint and the rank constraint. To relax the $\ell_0$ constraint, notice that $\normm{Z}_1 \leq s$ whenever $\Tr(Z)=1$ and $\normm{Z}_0\leq s^2$, since $\normm{Z}_1 \leq s\normm{Z}_2 = s \Tr(Z)^{1/2}=s \Tr(v^\top v)^{1/2}= s$. Thus, a convex relaxation of the $\ell_0$ constraint is given by $\normm{Z}_1 \leq s$. For the rank constraint, however, there is no obvious convex relaxation. By simply dropping the rank constraint, the optimization problem in \eqref{sparseeigalt} can be relaxed to the semidefinite program
\begin{align}
    \widehat{\lambda}_{\max}^s(A) &= \underset{\substack{Z \in N(p,s) }}{\sup} \  | \mathrm{Tr}(AZ) | \label{conveigdef},
\end{align}
where $ N(p,s) = \{Z \in \RR^{p \times p} \ ; \ Z \succcurlyeq 0, \mathrm{Tr}(Z)=1, \normm{Z}_1\leq s\}$. As a convex problem, it can be solved efficiently using e.g. interior point methods \citep[see][]{Boyd_2004} or first order methods \citep[see][]{bach2010convex}, the latter in polynomial time in $p$.

Due to the results in \cite{bach2010convex}, formalized in Lemma \ref{duallemma} in the Supplementary Material, we have
\begin{align}
    \lambda_{\max}^s (A) \leq  \widehat{\lambda}_{\max}^s(A) \leq s\normm{A}_{\infty},\label{convexcontrol}
\end{align}
which allows for high-probability control over the relaxed sparse eigenvalue of $\widehat{\Sigma}_{1,t} - \widehat{\Sigma}_{2,t}$. 

\subsection{The testing procedure}
We now modify the adaptive changepoint test from Section \ref{adaptivesec} by using the convex relaxation of the $s$-sparse largest absolute eigenvalue. Replacing the $s$-sparse largest absolute eigenvalue in \eqref{stsdef} with the convex relaxation defined in \eqref{conveigdef}, we define
\begin{align}
    \widehat{S}_{t,s} = \widehat{\lambda}_{\max}^s ( \widehat{\Sigma}_{1,t} - \widehat{\Sigma}_{2,t}) \label{hatstsdef},
\end{align}
where the $\widehat{\Sigma}_{i,t}$ for $i=1,2$ are defined in \eqref{sigmahats}. Since we will use \eqref{convexcontrol} to control the behavior of $\widehat{S}_{t,s}$, it suffices to estimate the noise level of $\widehat{S}_{t,s}$ by 
\begin{align}
    \widehat{\sigma}^2_{\mathrm{con}} &= \widehat{\lambda}^1_{\max} (\widehat{\Sigma}_{1,\lceil \log(ep)\rceil})  \wedge \widehat{\lambda}^1_{\max} (\widehat{\Sigma}_{2,\lceil \log(ep)\rceil}) \\
    &= \lVert\widehat{\Sigma}_{1,\lceil \log(ep)\rceil}\rVert_{\infty} \wedge \lVert\widehat{\Sigma}_{2,\lceil \log(ep)\rceil}\rVert_{\infty}\label{convsigmahat},
\end{align}
where $\widehat{\Sigma}_{i,t}$ is defined in \eqref{sigmahats} for $i=1,2$ and $t \in [n]$. 
We remark that $\widehat{\sigma}^2_{\mathrm{con}}$ uses the first and last $\lceil \log(ep) \rceil$ observations to estimate the order of the noise level of $\widehat{S}_{t,s}$, which implicitly restricts the sample size to satisfy $n \geq \log(ep)$. 

We define the computationally tractable variant of the testing procedure in \eqref{adaptivemultivariatetest} by
\begin{align} 
    \widehat{\psi}_{\mathrm{adaptive}}(X) &= \widehat{\psi}_{\mathrm{adaptive},\lambda}(X) = \underset{t \in \mathcal{T}}{\max} \ \underset{s \in \mathcal{S}}{\max} \ \ind\left\{       \widehat{S}_{t,s} > \lambda \widehat{\sigma}^2_{\mathrm{con}} h(p,n,s,t)\right\}, \ \ \ \label{computationaltest}
    \end{align}
    where $\widehat{S}_{t,s}$ is given in \eqref{hatstsdef},  $\mathcal{T}$ is given in \eqref{mathcalt} and $\mathcal{S}$ is given in \eqref{mathcals}, and
    \begin{align}
        h(p,n,s,t) = s \left\{   \sqrt{\frac{\log(ep) \vee \log\log(8n)}{t}} \vee \frac{\log(ep)\vee \log\log(8n)}{t}  \right\} \label{hdef}.
    \end{align}

The following theorem gives the theoretical performance of $ \widehat{\psi}_{\mathrm{adaptive}}$. 
\begin{proposition}\label{comuptationaltestresult}
Assume that $X_1, \ldots, X_n$ satisfy Assumption \ref{assmultivariate2} for some $w,u>0$. Let $\EE_0$ denote the expectation operator with respect to the distribution of $X= (X_1^\top, \ldots, X_n^\top)^\top$ when $\Cov(X_i)$ is constant, and let $\EE_1$ denote the expectation operator when $\Cov(X_i) = \Sigma_1$ for $i \leq t_0$ and $\Cov(X_i) = \Sigma_2$ for $i> t_0$, where $\Sigma_1$ and $\Sigma_2$ are some positive definite matrices.

    Given any $\delta \in (0,1)$,  there exists some $\lambda_0>0$ depending only on $\delta$, $w$ and $u$, such that the test in \eqref{computationaltest} with $\lambda \geq \lambda_0$ satisfies
    \begin{align}
         \EE_{0} \widehat{\psi}_{\mathrm{adaptive}}(X) +  \EE_{1} \left(1 - \widehat{\psi}_{\mathrm{adaptive}}(X)\right) \leq \delta
    \end{align}
    as long as
    \begin{align}
        \min(t_0,n-t_0)\left( \frac{\lambda_{\max}^s(\Sigma_1 - \Sigma_2)}{\lambda_{\max}^s(\Sigma_1) \vee \lambda_{\max}^s(\Sigma_1)} \right)^2 \geq  C s^2 \left\{ \log(ep) \vee \log\log(8n)\right\}, \ \ \ \label{prop6condition}
    \end{align}
    for some $s \in [p]$, where $C>0$ is a constant depending only on $\lambda, \delta, w$ and $u$. 
\end{proposition}

Comparing Proposition \ref{comuptationaltestresult} to Proposition \ref{adaptivemultivariateupperbound}, the computational feasibility of the testing procedure in \eqref{computationaltest} comes at the price of an increased signal strength required for guaranteed detection of a changepoint. To see this, recall that the computationally unfeasible testing procedure in \eqref{multivariatetest} has a signal strength requirement of order $s\log(ep/s) \vee \log\log(8n)$. In comparison, the computationally feasible variant in \eqref{computationaltest} has a signal strength requirement of order $s^2\{\log(ep)\vee \log\log(8n)\}$. However, when $p \geq \log n$, this discrepancy is likely to be unavoidable for any polynomial-time algorithm. Indeed, in the special case where the changepoint location is known, $\Sigma_1 = I$ and $\Sigma_2 = I + \theta v v^\top$ for some $\theta>0$ and $v \in S^{p-1}_s$, the covariance changepoint problem reduces to a sparse PCA problem. For this problem, \cite{Berthet13} showed that if \eqref{prop6condition} were to be improved by a randomized polynomial-time algorithm, then such an algorithm could be modified to detect a Planted Clique in a regime widely conjectured to be computationally infeasible \citep[see also][]{brennan2019}. For the case where $n > e^p$, however, it is unclear whether the rate $s^2 \log\log(8n)$ in \eqref{prop6condition} is improvable for polynomial-time algorithms.

\section{Proofs of main results}\label{proofs}
\subsection{Proofs of results from Section \ref{univariatesec}}
\begin{proof}[Proof of Proposition \ref{univariateupper}]
Fix any $\delta \in (0,1)$,  $\sigma>0$, and $w,u>0$.  
Write $X=(X_1,\ldots,X_n)^\top$ for the random vector in $\RR^{n}$ consisting of the observed data. 
    
We first control the Type I error under the null hypothesis. Assume that the distribution of $X$ belongs to the class $\mathcal{P}_0(n,w,u,\sigma)$ given in Equation \eqref{p0univariate}, so that $\Var(X_i) = \sigma^2$ for all $i \in [n]$. Consider the events 
\begin{align}
    \mathcal{E}_1 &= \bigcap_{i=1,2} \bigcap_{ t \in \mathcal{T}} \left\{ \widehat{\sigma}_{i,t}^2 \geq \sigma^2 c_1  \right\},\\
    \mathcal{E}_2 &= \bigcap_{i=1,2} \bigcap_{ t \in \mathcal{T}} \left\{  \left| \widehat{\sigma}_{i,t}^2 - \sigma^2\right| \leq \sigma^2  c_2 \left( \frac{\log\log(8n)}{t} \vee \sqrt{\frac{\log\log(8n)}{t}}\right)  \right\},
\end{align}
where $c_1 = \delta^2 (16 w)^{-2} (2e\pi)^{-1}$, $c_2 = c_0 \left\{ 2 + \log\left(8/\delta\right)\right\} \vee 1$ and $c_0$ is the absolute constant in Lemma \ref{covarianceconcentrationlemma}, depending only on $u$. Define $\mathcal{E} = \mathcal{E}_1 \cap \mathcal{E}_2$. Then Lemma \ref{eventbound1} implies that $\PP(\mathcal{E}) \geq 1 - \delta/2$. 

We claim that $S_t \leq \lambda \left( \frac{\log\log(8n)}{t} \vee \sqrt{\frac{\log\log(8n)}{t}}\right)$ for all $t \in \mathcal{T}$ on $\mathcal{E}$ whenever $\lambda >0$ is chosen sufficiently large. To see this, observe that
\begin{align}
    \frac{\widehat{\sigma}_{2,t}^2}{\widehat{\sigma}_{1,t}^2} - 1 &= \frac{\widehat{\sigma}_{2,t}^2 - \widehat{\sigma}_{1,t}^2}{\widehat{\sigma}_{1,t}^2}\\
    &\leq \frac{2c_2}{c_1} \left( \frac{\log\log(8n)}{t} \vee \sqrt{\frac{\log\log(8n)}{t}}\right),
\end{align}
on $\mathcal{E}$, for any $t \in \mathcal{T}$. Since the same holds true for ${\widehat{\sigma}_{1,t}^2} / {\widehat{\sigma}_{2,t}^2} - 1$, the claim holds by choosing $\lambda \geq 2c_2 /c_1$. Since the distribution of $X$ was arbitrarily chosen from $\mathcal{P}_0(n,w,u,\sigma)$, it follows that
\begin{align}
    \underset{P \in \mathcal{P}_0(n,w,u,\sigma) }{\sup} \ \EE_P \left(   \psi_{\lambda}(X)   \right) \leq \delta/2, 
\end{align}
for any such $\lambda$.

Next, we consider the Type II error. Let $\lambda$ remain unchanged, and assume now that the distribution of $X$ belongs to the class $\mathcal{P}(n,w,u,\rho)$ given in Equation \eqref{punivariate}, where $\rho \geq C \log\log(8n)$ and
\begin{align}
     C > \left\{4\lambda + 8c_0 \sqrt{\log(8/\delta)}\right\}^2 \vee  \frac{4c_3^2}{c_1^2} \left\{ \lambda + (4c_0 \vee 1) \sqrt{\log(8/\delta)}\right\}^2,
\end{align}
where $c_0$ is the absolute constant from Lemma \ref{covarianceconcentrationlemma}, depending only on $u$, while $c_1$ is as before and $c_3 = 1+ c_0 \log(8/ \delta)$. Note that $C$ only depends on $\lambda, \delta$, $w$ and $u$.  Let $t_0 \in [n-1]$, $\sigma_1^2$ and $\sigma_2^2$ respectively be the changepoint location and pre- and post-change variances of the $X_i$, so that $\Var(X_i) = \sigma_1^2$ for $i\leq t_0$ and $\Var(X_i) = \sigma_2^2$ for $i>t_0$, for some $\sigma_1^2, \sigma_2^2>0$.  Without loss of generality, we may assume that $\sigma^2_2 > \sigma_1^2$ and $t_0 \leq n-t_0$. 

By the definition of $\mathcal{T}$, there exists some $t \in \mathcal{T}$ such that $t_0/2 \leq t \leq t_0$. For this $t$, define the events 
\begin{align}
    \mathcal{E}_3 &= \left\{   \widehat{\sigma}_{2,t}^2 \geq  \sigma_2^2 \left( c_1 \vee  \left[1 -c_0\left\{ \frac{\log(8/\delta)}{t} \vee \sqrt{\frac{\log(8/\delta)}{t}}\right\}  \right]   \right)  \right\}\\
    \mathcal{E}_4 &= \left\{ \widehat{\sigma}_{1,t}^2 \leq \sigma_1^2 \left[ 1 + c_0\left\{ \frac{\log(8/\delta)}{t} \vee \sqrt{\frac{\log(8/\delta)}{t}}\right\} \right] \right\}, 
\end{align}
where $c_0$ and $c_1$ are as before. Now set $\mathcal{E} = \mathcal{E}_3 \cap \mathcal{E}_4$, and note that $\PP\left( \mathcal{E}\right) \geq 1 - \delta/2$ due to Lemma \ref{eventbound2}.  
We first show that $\psi_{\lambda}(X)=1$ on the event $\mathcal{E}$ whenever $t \geq (16 c_0^2 \vee 1) \log(8/\delta)$. On $\mathcal{E}$, we have
\begin{align}
    S_t &\geq \frac{\widehat{\sigma}_{2,t}^2}{\widehat{\sigma}_{1,t}^2} - 1\\
    &\geq \frac{\sigma_2^2 \left(    1 -c_0 \sqrt{\frac{\log(8/\delta)}{t}}  \right)}{ \sigma_1^2 \left( 1 + c_0\sqrt{\frac{\log(8/\delta)}{t}}  \right)} -1\\
    &\geq \frac{\sigma_2^2}{\sigma_1^2} \left( 1 - 2c_0 \sqrt{\frac{\log(8/\delta)}{t}}  \right) -1.
\end{align}
Now, since $\Cov(X) \in \Theta(n,\rho)$, we have that
    \begin{align}
        \frac{\sigma_2^2}{\sigma_1^2} - 1 &\geq C^{1/2}  \left( \frac{\log\log(8n)}{t_0} \vee \sqrt{\frac{\log\log(8n)}{t_0}}\right)\\
        &\geq \frac{1}{2} C^{1/2}  \left( \frac{\log\log(8n)}{t} \vee \sqrt{\frac{\log\log(8n)}{t}}\right)
    \end{align}
    where we used that $t \geq t_0/2$. It follows that
\begin{align}
    \frac{\sigma_2^2}{\sigma_1^2} &\geq 1 +  \frac{C^{1/2}}{2}  \left( \frac{\log\log(8n)}{t} \vee \sqrt{\frac{\log\log(8n)}{t}}\right),
\end{align}
and therefore 
\begin{align}
    S_t &\geq \left(  1 +  \frac{C^{1/2}}{2}  \left( \frac{\log\log(8n)}{t} \vee \sqrt{\frac{\log\log(8n)}{t}}\right)   \right) \left( 1 - 2c_0 \sqrt{\frac{\log(8/\delta)}{t}}  \right) -1.
    \intertext{On the event $\mathcal{E}$. Since $2 c_0 \sqrt{\frac{\log(8/\delta)}{t}} \leq 1/2$, we thus have}
    S_t &\geq  \frac{1}{4} C^{1/2}  \left( \frac{\log\log(8n)}{t} \vee \sqrt{\frac{\log\log(8n)}{t}}\right)  - 2c_0 \sqrt{\frac{\log(8/\delta)}{t}} . 
\end{align}
It follows that 
\begin{align}
    &S_t - \lambda  \left( \frac{\log\log(8n)}{t} \vee \sqrt{\frac{\log\log(8n)}{t}}\right) \\
    &\geq  \left( \frac{1}{4} C^{1/2} - \lambda\right)  \left( \frac{\log\log(8n)}{t} \vee \sqrt{\frac{\log\log(8n)}{t}}\right)  - 2c_0 \sqrt{\frac{\log(8/\delta)}{t}},\\
    &\geq \left( \frac{1}{4} C^{1/2} - \lambda -  2c_0 \sqrt{ \log(8/\delta)}\right) \frac{1}{\sqrt{t}} \\
    &>0,
\end{align}
on the event $\mathcal{E}$, where we used that $C^{1/2}/4 - \lambda>0$ in the second inequality and $C^{1/2} >  4\lambda + 8 c_0 \sqrt{\log(8/\delta)}$ in the third. We conclude that $\psi_{\lambda}(X)=1$ on $\mathcal{E}$ whenever $t \geq (16c_0^2 \vee 1) \log(8/\delta)$. 

Assume now that $t \leq (16c_0^2\vee 1) \log(8/\delta)$. On $\mathcal{E}$, we have
\begin{align}
    S_t  &\geq \frac{\widehat{\sigma}_{2,t}^2}{\widehat{\sigma}_{1,t}^2} - 1\\
    &\geq \frac{c_1}{c_3} \frac{\sigma_2^2}{\sigma_1^2} - 1,
    \end{align}
    where $c_3 = 1+ c_0 \log(8/ \delta)$. As before, we have
    \begin{align}
        \frac{\sigma_2^2}{\sigma_1^2} - 1 &\geq \frac{1}{2} C^{1/2}  \left( \frac{\log\log(8n)}{t} \vee \sqrt{\frac{\log\log(8n)}{t}}\right),
    \end{align}
    and thus
    \begin{align}
     &S_t  - \lambda \left( \frac{\log\log(8n)}{t} \vee \sqrt{\frac{\log\log(8n)}{t}}\right) \\
     &\geq \left (  \frac{c_1}{2c_3}  C^{1/2} - \lambda\right)   \left( \frac{\log\log(8n)}{t} \vee \sqrt{\frac{\log\log(8n)}{t}}\right) -1 \\
     &\geq  \left (  \frac{c_1}{2c_3}  C^{1/2} - \lambda\right)t^{-1/2}  -1 \\
     &\geq \left (  \frac{c_1}{2c_3}  C^{1/2} - \lambda\right)\left\{ (16c_0^2\vee 1) \log(8/\delta) \right\}^{-1/2}  -1  \label{ineq1},
\end{align}
where we in the second inequality used that $\log \log(8n)\geq 1$, and last inequality used that $c_1 C^{1/2} /(2c_3) - \lambda>0$ and $t \leq (16c_0^2\vee 1) \log(8/\delta)$. Due to the definition of $C$, the right-hand side of \eqref{ineq1} is positive, and therefore $\psi_{\lambda}(X)=1$ holds on the event $\mathcal{E}$ also when $t \leq (16 c_0^2 \vee 1) \log(8/\delta)$.

Since the distribution of $X$ was chosen arbitrarily from $\mathcal{P}(n,w,u,\rho)$, it follows that
\begin{align}
    \underset{P \in \mathcal{P}(n,w,u,\rho) }{\sup} \EE_{P} \left(1 - \psi_{\lambda}(X)\right) \leq \delta/2, 
\end{align}
and the proof is complete.

\end{proof}
\begin{proof}[Proof of Proposition \ref{univariatelower}]
The proof of Proposition \ref{univariatelower} adopts techniques from \cite{liu_minimax_2021}. Fix any $\sigma>0, w \geq (2\pi)^{-1/2}$ and $u>0$ sufficiently large. To prove Proposition \ref{univariatelower} we will impose a prior on the alternative hypothesis parameter space and bound the total variation distance between the null distribution and the mixture distribution induced by the prior. Since $w \geq (2\pi)^{-1/2}$ and $u$ is sufficiently large, we can take the distribution of the $X_i$ to be Gaussian. Given any $0 \leq \eta \leq 1$, it then suffices by Lemma \ref{minimaxlemma} to find an absolute constant $c>0$ and a probability measure $\nu$ with $\text{supp}(\nu) \subseteq \Theta(n, c\log\log(8n))$ such that
$\chi^2(f_1,f_0) \leq \eta,$ where $f_0$ and $f_1$ are the densities of $X \sim \text{N}_n(0,V)$ when $V= \sigma^2I$ and $V \sim \nu$, respectively. 

Fix $0<c\leq 1$, to be specified later, and set $\rho = c \log \log(8n)$. Define $\nu$ to be the distribution of $V\in \RR^{n\times n}$ generated according to the following process: 
\begin{enumerate}
    \item Let $\Delta = 2^l$, where $l$ is sampled uniformly from the set \\ $\left\{0, 1, 2,\ldots, {\lfloor \log_2(n/2)\rfloor\}}\right\}$. \label{unisampstep1}
    \item Given $\Delta$, set $V = V(\Delta) = \text{Diag}(\left\{ V_j \right\}_{j=1}^n )$, where $V_j = \sigma_1^2 = \sigma^2  - \kappa$ for $j \leq \Delta$ and $V_j = \sigma_2^2 = \sigma^2$ for $j > \Delta$, where $\kappa = \kappa(\Delta)$ is defined by
    \begin{align}
        \kappa &= \begin{cases}
            \sigma^2 \frac{ \rho }{\Delta + \rho}, &\text{ if } \Delta \leq \rho \\
            \sigma^2 \frac{ \sqrt{ \rho} }{\sqrt{\Delta} + \sqrt{\rho}},&\text{ otherwise.}
            \end{cases}
    \end{align} \label{unisampstep2}
\end{enumerate}
In the sampling process above, $\sigma_1^2 < \sigma_2^2$ are the pre- and post-change variances, respectively. One can easily verify that
\begin{align}
    \min(\Delta, n - \Delta) \left\{ \left( \frac{\left| \sigma_1^2 - \sigma_2^2\right| }{ \sigma_1^2 \wedge \sigma^2_2} \right) \wedge \left( \frac{\left| \sigma_1^2 - \sigma_2^2\right| }{ \sigma_1^2 \wedge \sigma^2_2}\right)^2 \right\}= \rho, \label{snrok}
\end{align}
given any $\Delta$, where $\sigma_1^2 = \sigma_1^2(\Delta)$,  and hence $\text{supp}(\nu) \subseteq \Theta(n,\rho)$. 

Let $f_0$ denote the density of $\text{N}_n(0, \sigma^2I_n)$ and $f_1$ denote the density of the mixture distribution $\text{N}_n(0, V)$ when $V \sim \nu$. For any positive definite matrix $M$, let $\phi_{M}(\cdot)$ denote the density of $\text{N}(0, M)$, so that $f_0(\cdot) = \phi_{\sigma^2I}(\cdot)$ and $f_1(\cdot) = \EE_{V \sim \nu}\left\{ \phi_{V}(\cdot)\right\}$. We have that
\begin{align}
    \chi^2(f_1, f_0) +1 &= \EE_{x \sim f_0} \left\{   \frac{f_1(x)^2}{f_0(x)^2}      \right\}\\
    &=  \EE_{x \sim f_0} \left[   \frac{\left[ \EE_{V \sim \nu}\left\{ \phi_{V}(x)\right\}\right]^2}{\phi^2_{\sigma^2I}(x)}      \right]\\
    &= \EE_{x \sim f_0} \left[   \frac{ \EE_{(V_1, V_2) \sim \nu \otimes \nu}\left\{ \phi_{V_1}(x) \phi_{V_2}(x)\right\}}{\phi^2_{\sigma^2I}(x)}      \right].
\end{align}
Here, $V_1 = V(\Delta_1)$ and $V_2 = V(\Delta_2)$ denote independent samples from $\nu$. By Fubini's Theorem we thus have
\begin{align}
    \chi^2(f_1, f_0) +1 &= \EE_{(V_1, V_2) \sim \nu \otimes \nu} \EE_{x \sim f_0} \left\{   \frac{\phi_{V_1}(x) \phi_{V_2}(x)}{\phi^2_{\sigma^2I}(x)}      \right\}. 
\end{align}
By Lemma \ref{lemmachisqbound}, it follows that
\begin{align}
    \chi^2(f_1, f_0) +1 &\leq \EE_{(V_1, V_2) \sim \nu \otimes \nu} \exp\left[   \frac{1}{2}\left\{  \sqrt{\frac{\Delta_-}{\Delta_+}} (\Delta_+ \alpha_+^2)^{1/2} (\Delta_- \alpha_-^2)^{1/2}  \wedge   \Delta_- \alpha_- \right\} \right],
\end{align}
where $\alpha_i = \kappa_i (\sigma^2-\kappa_i)^{-1}$, using subscripts $+$ and $-$ to indicate the index of the largest and smallest value of the $\Delta_i$, respectively, for $i=1,2$. Here, $(\Delta_i, \kappa_i)$ are there variables from from Step \ref{unisampstep1} and \ref{unisampstep2} for generating $V_i$ for $i=1,2$. Due to the definition of $\kappa$, we have that $\Delta_- \alpha_- = \rho$ whenever $\Delta_- \leq \rho$. If on the other hand $\Delta_- \geq \rho$, then we also have $\Delta_+ \geq \rho$, in which case $\Delta_i \alpha_i^2 =  \rho$ for $i=1,2$. Hence, 
\begin{align}
    \chi^2(f_1, f_0) +1 &\leq \EE_{(V_1, V_2) \sim \nu \otimes \nu} \exp\left(   \frac{1}{2}  \rho \right) \ind\left \{  \Delta_- \leq  \rho \right\}  \label{boundc}
    \\ &+ \EE_{(V_1, V_2) \sim \nu \otimes \nu} \exp\left(   2^{- |l_1 - l_2|/2 -1}  \rho \right),
\end{align}
where $l_i = \log_2(\Delta_i)$ for $i=1,2$. We bound the two terms on the right-hand side of \eqref{boundc} separately. For the first term, we have
\begin{align}
\EE_{(V_1, V_2) \sim \nu \otimes \nu} \exp\left(   \frac{1}{2} \rho \right) \ind\left \{  \Delta_- \leq  \rho \right\} &\leq \log(8n)^{c/2} \PP\left[l_1 \leq \log_2\left\{  c\log\log(8n)\right\} \right]. \ \ \label{boundfirst}
\end{align}
Now choose $c \leq \log(16/\eta)^{-1} /2 \wedge 1$. Then if $\log \log(8n) < c^{-1}$, the right-hand side of \eqref{boundfirst} is zero since $c \log\log(8n)<1$. If we instead have $\log\log(8n)\geq c^{-1}$, the right-hand size of \eqref{boundfirst} is bounded above by 
\begin{align}
&\quad \ \ \ \log(8n)^{c/2} \frac{ 1+ \left \lfloor  \log_2\{ c \log \log(8n) \}  \right \rfloor}{1 + \left \lfloor \log_2(n/2)\right \rfloor}\\
&\leq \left[ \underset{n \in \NN; \ n\geq 2}{\sup} \ \frac{ \log(8n) ^{1/2}\left\{ 1+ \left \lfloor  \log_2  \log \log(8n) \right \rfloor\right\}}{1 + \left \lfloor \log_2(n/2)\right \rfloor} \right]  \log(8n)^{(c-1)/2} \\
&\leq 2 \log(8n)^{(c-1)/2}\\
&= 2 \exp \left\{ \frac{c-1}{2} \log\log(8n)\right\}\\
&\leq 2 \exp \left( \frac{c-1}{2c}\right) \leq 4 \exp \left( - \frac{1}{2c}\right) \leq \eta/4,
\end{align}
where we used the fact that $c \leq \log(16/\eta)^{-1}$ in the last inequality.

Now we bound the second term in the right-hand side of \eqref{boundc}. We have
\begin{align}
    &\EE_{(V_1, V_2) \sim \nu \otimes \nu} \exp\left(   2^{- |l_1 - l_2|/2 -1} \rho \right) \label{bound2} \\
    &= \EE_{l_1, l_2} \exp\left\{   \frac{c}{2} \log \log(8n)\right\} \ind \left\{ \left| l_1 - l_2\right| =0  \right\}\\
    &+ \EE_{l_1, l_2} \exp\left\{   2^{- |l_1 - l_2|/2 -1} c \log \log(8n)\right\} \ind \left\{ 0 < \left| l_1 - l_2\right| \leq (\eta/18) \log\log(8n)  \right\}\\
    &+ \EE_{l_1, l_2} \exp\left\{   2^{- |l_1 - l_2|/2 -1} c \log \log(8n)\right\} \ind \left\{ \left| l_1 - l_2\right| > (\eta/18) \log\log(8n)  \right\}.
\end{align}
Note that $l_1,l_2$ are sampled independently and uniformly from an integer interval with cardinality $a(n) = \left \lfloor \log_2(n/2) \right\rfloor +1$, from which it follows that $\PP(|l_1 - l_2| = x)\leq 2  a(n)^{-1}$ for any $x \in \NN\cup \{0\}$. We claim that the first term at the right-hand side of \eqref{bound2} is bounded above by
\begin{align}
    \left\{ \left( 1 + \frac{\eta}{4} \right)  \PP_{l_1, l_2}(|l_1 - l_2| = 0)\right\} \vee \frac{\eta}{4},
\end{align}
provided that $c$ is lowered (if necessary) to satisfy $c \leq \eta \log(1+\eta/4)/6 \wedge 1$. Indeed, for $n \geq \exp \{ \exp(12/\eta)\}/8$, we have
\begin{align}
    \EE_{l_1, l_2} \exp\left\{   \frac{c}{2} \log \log(8n)\right\} \ind \left\{ \left| l_1 - l_2\right| =0  \right\} &\leq \frac{\log(8n)^{c/2}}{1 + \lfloor \log_2(n/2)\rfloor} \\
    &\leq \frac{\eta}{12} \ \underset{n \in \NN ; \ n\geq 2}{\sup} \ \frac{\log(8n)^{1/2}\log\log(8n)}{1 + \lfloor \log_2(n/2)\rfloor} \\
    &\leq \frac{\eta}{4}.
\end{align}
If we instead have $n < \exp \{ \exp(12/\eta)\}/8$, then 
\begin{align}
    &\EE_{l_1, l_2} \exp\left\{   \frac{c}{2} \log \log(8n)\right\} \ind \left\{ \left| l_1 - l_2\right| =0  \right\} \\
    =& \exp\left\{ \frac{c}{2} \log\log(8n)\right\} \PP_{l_1, l_2}( |l_1 - l_2| = 0)\\
    \leq & \exp\left\{ \frac{6c}{\eta} \right\} \PP_{l_1, l_2}( |l_1 - l_2| = 0)\\
    \leq & \left( 1 + \frac{\eta}{4}\right) \PP_{l_1, l_2}( |l_1 - l_2| = 0),
\end{align}
using that $c \leq \eta \log(1+\eta/4)/6\wedge 1$ in the last inequality.

For the second term at the right-hand side of \eqref{bound2}, using that $c \leq 1$, we have
\begin{align}
    &\EE_{l_1, l_2}  \exp \left( 2^{-|l_1 - l_2|/2-1} c\log \log(8n)  \right) \ind\left\{ 0 < \left| l_1 - l_2\right| \leq (\eta/18) \log \log (8n)\right\} \\
    \leq &  \log(8n)^{1/2} \PP_{l_1, l_2}\left\{ 0 < \left| l_1 - l_2\right| \leq (\eta/18) \log \log (8n) \right\}  \\
    \leq & (\eta / 9) \underset{n \in \NN; n\geq 2}{\sup} \ \frac{ \log \log (8n) \log(8n)^{1/2}}{a(n)} \\
    \leq & \eta/4.
\end{align}
Now let $b_{\eta} = \ \underset{n \geq 2}{\sup} \ \frac{\log \log(8n)}{\log(8n)^{\eta \log(2)/18}} <\infty$. Then for the second term on the right-hand side of \eqref{bound2}, we have
\begin{align}
    &\EE_{l_1, l_2}  \exp \left( 2^{-|l_1 - l_2|/2}\frac{c}{2} \log \log(8n)  \right) \ind\left\{\left| l_1 - l_2\right| > (\eta/18) \log \log (8n)\right\} \\
    \leq &  \exp\left(  \frac{c b_{\eta}}{2} \right) \PP_{l_1, l_2} \left\{\left| l_1 - l_2\right| > (\eta/18) \log \log (8n)\right\}\\
    \leq & \left(  1+ \frac{\eta}{4}\right) \PP_{l_1, l_2} \left\{\left| l_1 - l_2\right| > (\eta/18) \log \log (8n)\right\},
\end{align}
by choosing $c \leq 2\frac{\log(1+\eta/4)}{b_{\eta}}$. It follows that $\chi^2(f_1, f_0) \leq \eta$ whenever
\begin{align}
    c \leq \log(12/\eta)^{-1}/2 \wedge  \eta \log(1+\eta/4)/6  \wedge 2 \frac{\log\left(1 + \eta/4\right)}{b_{\eta}} \wedge 1,
\end{align}
and the proof is complete. 

\end{proof}
\subsection{Proofs of results from Section \ref{multivariatesec}}
\begin{proof}[Proof of Proposition \ref{multivariatelowerbound}]
We use a similar strategy as in the proof of Proposition \ref{univariatelower}. Fix any $\sigma>0$, $w \geq (2\pi)^{-1/2}$ and $u>0$ sufficiently large. Since $w \geq (2\pi)^{-1/2}$ and $u$ is sufficiently large, we can take the distribution of the $X_i$ to be Gaussian. Given any $0<\eta<1$, it suffices by Lemma \ref{minimaxlemma} to find an absolute constant $c>0$ and a probability measure $\nu$ with $\text{supp}(\nu) \subseteq \Theta(p,n, s,\sigma, c \{ s \log(ep/s)\vee \log\log(8n)\})$ such that
$\chi^2(f_1,f_0) \leq  \eta,$ where $f_0$ and $f_1$ are the densities of $X \sim \text{N}_{np}(0, V)$ with some $V \in \Theta_0(p,n,\sigma)$ fixed and $V \sim \nu$, respectively. 

Fix $c\leq 1$, to be specified later, and set $\rho = \rho(p,n,s) = c \{ s \log \left(\frac{ep}{s}\right) \vee \log \log(8n)\}$. Define $\nu$ as the distribution of $V$ generated according to the following process: 
\begin{enumerate}
    \item Let $\Delta = 2^l$, where $l$ is sampled uniformly from the set  \\
    $\left\{0,1,2, \ldots, \lfloor \log_2(n/2)\rfloor\right\}.$  \label{sampstepmulti1}
    \item Independently of $\Delta$, uniformly sample a subset $I \subseteq [p]$ with cardinality $s$. \label{sampstepmulti12}
    \item Independently of $\Delta$, sample $u = (u(1), \ldots, u(p))^\top \in \mathcal{S}^{p-1}_s$, where \\$u(i) \iid \text{Unif}\left\{ -s^{-1/2}, s^{-1/2}\right\}$ for $i \in I$ and $u(i)=0$ otherwise, and $I$ is the subset sampled in Step \ref{sampstepmulti12}.\label{sampstepmulti2}
    \item Given \label{sampstep3}$(\Delta,u)$, set $V = V(\Delta, u) = \text{Diag}( \left\{ V_j \right\}_{j=1}^n )$, where $V_j = \Sigma_1 = \sigma^2 I - \kappa u u^\top$ for $j \leq \Delta$ and $V_j = \Sigma_2 = \sigma^2 I$ for $j>\Delta$, where $\kappa = \kappa(\Delta,u)$ is defined by
    \begin{align}
        \kappa &= \begin{cases}
            \sigma^2 \frac{\rho }{\Delta + \rho}, &\text{ if } \Delta \leq \rho \\
            \sigma^2 \frac{ \sqrt{ \rho} }{\sqrt{\Delta} + \sqrt{\rho}},&\text{ otherwise.}
            \end{cases}
    \end{align} \label{sampstepmulti3}
\end{enumerate}
One can easily check that $\Sigma_1 = \sigma^2 I - \kappa u u^\top$ and $\Sigma_2 = \sigma^2 I$ sampled according to the above process are symmetric and positive definite, both with operator norm $\sigma^2$. Moreover, since $u^\top (\Sigma_2 - \Sigma_1) u = u^\top (\kappa u u ^\top)u = \kappa = \normmop{\Sigma_1 - \Sigma_2}$, and $u \in S^{p-1}_s$, the change in covariance is $s$-sparse. Moreover, the definition of $\kappa$ implies that 
\begin{align}
\min(t_0, n - t_0)  \left \{ \left( \frac{\normmop{\Sigma_1 - \Sigma_2}}{\sigma^2 - \normmop{\Sigma_1 - \Sigma_2}}  \right) \wedge \left( \frac{\normmop{\Sigma_1 - \Sigma_2}}{\sigma^2 - \normmop{\Sigma_1 - \Sigma_2}}\right)^2 \right\} = \rho,
\end{align}
given any $\Delta$, and hence $\text{supp}(\nu) \subseteq \Theta(p,n,s, \sigma,\rho)$.

Let $f_0$ denote the density of $\text{N}_{np}(0, \sigma^2 I)$ and let $f_1$ denote the density of the mixture distribution $\text{N}_{np}(0, V)$ induced by sampling $V \sim \nu$.  As in the proof of Proposition \ref{univariatelower}, Lemma \ref{lemmachisqbound} implies that
\begin{align}
    &\chi^2(f_1, f_0) +1 \\
    &\leq \EE_{(V_1, V_2) \sim \nu \otimes \nu} \exp\left[   \frac{1}{2} \inner{u_1}{u_2}^2 \left\{  \sqrt{\frac{\Delta_-}{\Delta_+}} (\Delta_+ \alpha_+^2)^{1/2} (\Delta_- \alpha_-^2)^{1/2}  \wedge   \Delta_- \alpha_- \right\} \right],
\end{align}
where $\alpha_i = \kappa_i (\sigma^2-\kappa_i)^{-1}$, using subscripts $+$ and $-$ to indicate the index of the largest and smallest value of the $\Delta_i$, respectively, for $i=1,2$. Here, $(\Delta_i, \kappa_i)$ are there variables from from Step \ref{unisampstep1} and \ref{unisampstep2} for generating $V_i$ for $i=1,2$.

If $\rho = c \log\log(8n)$, one can apply the same steps as in the proof of Proposition \ref{univariatelower} to obtain $\chi^2(f_1, f_0) \leq \eta$ by choosing $c$ sufficiently small (depending only on $\eta$), since $\inner{u_1}{u_2}^2 \leq 1$. On the other hand, if $\rho =  c s \log \left( ep/s\right)$, observe that $\Delta_- \alpha_- = \rho$ whenever $\Delta_- \leq \rho$ and $\Delta_i \alpha_i^2 =  \rho$ for $i=1,2$ otherwise, as $\Delta_- \geq \rho$ implies $\Delta_+ \geq \rho$. It follows that
\begin{align}
    \chi^2(f_1, f_0) +1 &\leq \EE_{(V_1, V_2) \sim \nu \otimes \nu} \exp\left(  \frac{1}{2} \inner{u_1}{u_2}^2 \rho \right)\\
     &= \EE_{u_1, u_2}  \exp \left\{ \frac{c}{2s}  \log \left( \frac{ep}{s}\right) s^2 \inner{u_1}{u_2}^2\right\}.
\end{align}

The distribution of $s\inner{u_1}{u_2}$ equals that of $\sum_{i=1}^H R_i$, where the $R_i$ are independent Rademacher random variables and $H$ is a Hypergeometric random variable, independent of the $R_i$,  with parameters $(p,s,s)$. By lowering $c$ if necessary, it follows from Lemma \ref{rademacherlemma} that $ \chi^2(f_1, f_0)< \eta$, and the proof is complete. 

\end{proof}

\begin{proof}[Proof of Proposition \ref{multivariateupperbound}]
    Fix any $\delta \in (0,1)$, $\sigma>0$, $s \in [p]$ and $u>0$. 
    Write $X = (X_1^\top, \ldots, X_n^\top)^\top$ for the random vector in $\RR^{pn}$ consisting of the observed data. 
    
    We first control the Type I error under the null hypothesis. Assume that the distribution of $X$ belongs to the class $\mathcal{P}_0(p,n,u,\sigma)$ given in Equation \eqref{p0multivariate}, so that $\Cov(X_i) = \Sigma$ for all $i \in [n]$ and some $\Sigma$ with operator norm $\normmop{\Sigma} = \sigma^2$. Note first that 
    \begin{align}
        \lambda_{\max}^s({\widehat{\Sigma}_{1,t} - \widehat{\Sigma}_{2,t}}) &\leq \lambda_{\max}^s ({\widehat{\Sigma}_{1,t} - \Sigma}) + \lambda_{\max}^s ({\widehat{\Sigma}_{2,t} - \Sigma}),\label{prop4firsteq}
    \end{align}
    for any $t \in \mathcal{T}$. Using Lemma \ref{covarianceconcentrationlemma}, we have
    \begin{align}
        \PP \left\{ \underset{i=1,2}{\max} \   \lambda_{\max}^s ({\widehat{\Sigma}_{i,t}- \Sigma}) \geq c_0 c \sigma^2 r(p,n,s,t)\right\} 
        &\leq 2 \exp\left\{-c \log\log(8n)\right\},
    \end{align}
    for any $c\geq 1$ and $t \in \mathcal{T}$, where $r(p,n,s,t)$ is given in Equation \eqref{rdef} and $c_0$ is the constant from Lemma \ref{covarianceconcentrationlemma} depending only on $u$.  If $c \geq 2 + \log(4/\delta)$, then by a union bound, the event
    \begin{align}
        \mathcal{E} =  \bigcup_{t \in \mathcal{T}} \left\{ \lambda_{\max}^s({\widehat{\Sigma}_{1,t} - \widehat{\Sigma}_{2,t}}) < 2 c_0 c \sigma^2 r(p,n,s,t) \right\}
    \end{align}
    occurs with probability at least
    \begin{align}
        &\PP  (\mathcal{E}) \geq 1 - \PP (\mathcal{E}^\complement)\\
        &\geq 1-  2 \frac{\left| \mathcal{T}\right|}{\log(8n)^c}\\
        &\geq 1 - 2 \frac{1 + \lfloor \log_2(n/2)\rfloor}{\log(8n)^c}\\
        &\geq 1 - \delta/2, 
    \end{align}
    where the last inequality follows from the fact that 
    \begin{align}
        c &\geq 2 + \log(4/\delta) \\
&> \underset{n \in \NN; \ n\geq 2}{\sup} \ \frac{\log\left\{ 1 + \lfloor \log_2(n/2)\rfloor \right\} + \log(4/\delta)}{\log\log(8n)}.
    \end{align}
    Now choose $\lambda \geq \lambda_0= 2c_0\left\{2 +  \log(4/\delta)\right\}$. On $\mathcal{E}$, we then have
    \begin{align}
        S_{t,s} - \lambda \sigma^2 r(p,n,s,t)&\leq \left[ 2c_0\left\{2 +  \log(4/\delta)\right\} - \lambda \right] \sigma^2 r(p,n,s,t)\leq 0.
    \end{align}
    It follows that 
\begin{align}
     \EE \left(   \psi_{\lambda}(X)   \right) &\leq 
     \PP (\mathcal{E}^\complement) \leq \delta/2.
\end{align}

Next, we consider the Type II error. Let $\lambda$ remain unchanged, and assume now that the distribution of $X$ belongs to the class $\mathcal{P}(p,n,s,u,\sigma,\rho)$ given in Equation \eqref{pmultivariate}, where $\rho \geq C \{ s\log(ep/s)\vee \log\log(8n)\}$ and $C \geq 8\lambda^2$. Let $t_0 \in [n-1]$, $\Sigma_1$ and $\Sigma_2$ respectively be the changepoint location and pre- and post-change covariances of the $X_i$, so that $\Cov(X_i) = \Sigma_1$ for $i\leq t_0$ and $\Cov(X_i) = \Sigma_2$ for $i>t_0$, and $\Sigma_1,\Sigma_2$ are some positive definite matrices satisfying $\lambda_{\max}^s(\Sigma_1 - \Sigma_2) = \normmop{\Sigma_1 - \Sigma_2}$ and $\normmop{\Sigma_1}\vee \normmop{\Sigma_2}\leq \sigma^2$. Assume further that Equation \eqref{prop4condition} is satisfied for $t_0, \Sigma_1$ and $\Sigma_2$. Without loss of generality, we may assume that $\min(t_0, n-t_0) = t_0$.  

Note first that there exists some $t \in \mathcal{T}$ such that $t_0/2 \leq t \leq t_0$. For this $t$, we have
    \begin{align}
          \lambda_{\max}^s ({\widehat{\Sigma}_{1,t} - \widehat{\Sigma}_{2,t}}) &\geq \lambda_{\max}^s ({\Sigma_1 - \Sigma_2})- 2 \  \underset{i = 1,2}{\max} \  \lambda_{\max}^s ({\widehat{\Sigma}_{i,t} - \Sigma_i}),
    \end{align}
    Using similar arguments as above to bound the two terms at right-hand side of Equation \eqref{prop4firsteq}, the event 
    \begin{align}
        \mathcal{E} &= \left\{ \underset{i = 1,2}{\max} \  \sigma^{-2}\lambda_{\max}^s ({\widehat{\Sigma}_{i,t} - \Sigma_i}) < \frac{\lambda}{2} r(p,n,s,t) \right\}
    \end{align}
    has probability at least $1-\delta/2$. 

    Now, due to the assumption in Equation \eqref{prop4condition}, we have
    \begin{align}
        \sigma^{-2} \lambda_{\max}^s({\Sigma_1 - \Sigma_2} ) &\geq C^{1/2}   \sqrt{ \frac{s\log\left(\frac{ep}{s}\right) \vee \log\log(8n)}{t_0}}\\
        &\geq \sqrt{\frac{C}{2}}\sqrt{ \frac{s\log\left(\frac{ep}{s}\right) \vee \log\log(8n)}{t}}
    \end{align}
    Note also that since $C\geq 2$, Equation \eqref{prop4condition} implies that $t \geq s\log(ep/s)\vee \log\log(8n)$. To see this, note that $\lambda_{\max}^s({\Sigma_1 - \Sigma_2}) \leq \lambda_{\max}^p({\Sigma_1 - \Sigma_2}) \leq \normmop{\Sigma_1} \vee \normmop{\Sigma_2}\leq \sigma^2$, since $\Sigma_1$ and $\Sigma_2$ are positive definite. Since also $s\log(ep/s)\vee \log\log(8n) >1$, and $t \geq t_0/2$, Equation \eqref{prop4condition} cannot be true if $t < s\log(ep/s)\vee \log\log(8n)$. It follows that
     \begin{align}
         \lambda_{\max}^s({\Sigma_1 - \Sigma_2}) &\geq \sigma^2 \sqrt{\frac{C}{2}} r(p,n,s,t).
    \end{align}
    On the event $\mathcal{E}$, using that $C \geq 8 \lambda^2$ we thus have
    \begin{align}
        S_{t,s} - \lambda \sigma^2 r(p,n,s,t) &\geq \left( \sqrt{\frac{C}{2}} - \lambda \right) \sigma^2 r(p,n,s,t)> 0,
        \intertext{which implies that}
        \EE \left(1 - \psi_{\lambda}(X)\right)&\leq \PP(\mathcal{E}^\complement)\\
        &\leq \frac{\delta}{2},
    \end{align}
    and the proof is complete.
\end{proof}

\begin{proof}[Proof of Proposition \ref{adaptivemultivariateupperbound}]
    Fix any $\delta>0$, and assume that the observations $X_1, \ldots, X_n$ satisfy Assumption \ref{assmultivariate2} for some $w,u>0$. Write $X = (X_1^\top, \ldots, X_n^\top)^\top$ for the random vector in $\RR^{pn}$ consisting of the observed data. 
    
    We first control the Type I error under the null hypothesis. Assume that $\Cov(X) = \mathrm{Diag}( \left\{\Sigma\right\}_{i\in[n]})$ for some  positive definite covariance matrix $\Sigma \in \RR^{p \times p}$. Define the events
    \begin{align}
    \mathcal{E}_5 &= \bigcap_{\substack{s \in \mathcal{S};\\ \gamma(s)\leq n}} \left\{ c_4 \leq \frac{\widehat{\sigma}^2_s}{\lambda_{\max}^s({\Sigma})} \right\},\\
    \mathcal{E}_6 &=  \bigcap_{t \in \mathcal{T}} \bigcap_{s \in \mathcal{S}} \left\{ \lambda_{\max}^s ( {\widehat{\Sigma}_{1,t} - \widehat{\Sigma}_{2,t}} ) \leq c_5 \lambda_{\max}^s({\Sigma}) r(p,n,s,t) \right\},
    \end{align}
    where $\widehat{\sigma}_s^2$ is given in \eqref{sigmahatsdef}, $\mathcal{S}$ is given in \eqref{mathcals}, $r(p,n,s,t)$ is given in \eqref{rdef}, $c_4 =  \delta^2 (\delta^2 + 16\delta + 64)^{-1}(2e\pi w^2)^{-1}$,  $c_5 = 4c_0 \log(1 + 16/\delta)$, and $c_0$ is the constant from Lemma \ref{covarianceconcentrationlemma} depending only on $u>0$.  Define the event $\mathcal{E} = \mathcal{E}_5 \cap \mathcal{E}_6$. By Lemma \ref{eventbound3} we have $\PP(\mathcal{E})\geq 1 - \delta/2$.

    On $\mathcal{E}$, we have
    \begin{align}
        \frac{S_{t,s}}{\widehat{\sigma}^2_s} &\leq \frac{c_5 \lambda_{\max}^s ({\Sigma}) r(p,n,s,t)}{c_4\lambda_{\max}^s ({\Sigma})}\\
        &= \frac{c_5}{c_4}r(p,n,s,t),
    \end{align}
    for all $t\in \mathcal{T}$ and $s \in \mathcal{S}$. Choosing $\lambda \geq \lambda_0=c_5/c_4$, we thus obtain 
    \begin{align}
        \EE \phi_{\mathrm{adaptive},\lambda}(X) &\leq \PP(\mathcal{E}^\complement)\\
        &\leq \frac{\delta}{2}. 
    \end{align}

     Next we consider the Type II error. Let $\lambda$ remain unchanged, and assume now that $\Cov(X_i) = \Sigma_1$ for $i\leq t_0$ and $\Cov(X_i) = \Sigma_2$ for $i>t_0$, where $t_0 \in [n-1]$ and $\Sigma_1, \Sigma_2$ are some positive definite matrices in $\RR^{p\times p}$. 
     Assume further that $\Sigma_1, \Sigma_2$ and $t_0$ satisfy Equation \eqref{prop5condition} for some $s_0\in [p]$, where the constant $C$ satisfies
    \begin{align}
        C > 32 (\lambda c_3 + c_5)^2,
    \end{align}
    where $c_3$ is defined below and $c_5$ is as before.
    
    Without loss of generality, we may assume that $\min(t_0, n-t_0) = t_0$.  Note that there exists some $t \in \mathcal{T}$ such that $t_0/2 \leq t \leq t_0$ and some $s \in \mathcal{S}$ such that $s_0/2\leq s\leq s_0$ and $s \log(ep/s) \leq n$. Fixing this $s$ and $t$, we define the events 
    \begin{align}
    \mathcal{E}_7 &=  \left\{\frac{\widehat{\sigma}^2_s}{\lambda_{\max}^s(\Sigma_1) \wedge \lambda_{\max}^s(\Sigma_2)} \leq c_3\right\},\\
    \mathcal{E}_8 &=  \bigcap_{i=1,2} \left\{ \lambda_{\max}^s ( {\widehat{\Sigma}_{i,t} - \Sigma_i}) \leq \frac{c_5}{2} \lambda_{\max}^s({\Sigma_i}) r(p,n,s,t) \right\},
    \end{align}
    where $c_3 = 1 + c_0 \log(1+4/\delta)$ and $c_5$ is  as before. Define $\mathcal{E} = \mathcal{E}_7 \cap \mathcal{E}_8$. Due to Lemma \ref{eventbound4}, we have that $\PP(\mathcal{E})\geq 1 - \delta/2$. On $\mathcal{E}$, we have
    \begin{align}
          \lambda_{\max}^s ({\widehat{\Sigma}_{1,t} - \widehat{\Sigma}_{2,t}}) &\geq \lambda_{\max}^s({\Sigma_1 - \Sigma_2})- 2 \  \underset{i = 1,2}{\max} \  \lambda_{\max}^s({\widehat{\Sigma}_{i,t} - \Sigma_i})\\
          &\geq \lambda_{\max}^s({\Sigma_1 - \Sigma_2}) - c_5 \left\{ \lambda_{\max}^s(\Sigma_1) \vee \lambda_{\max}^s(\Sigma_2)\right\} r(p,n,s,t).
    \end{align}
    Now, due to the definition of $s_0/2 \leq s \leq s_0$ and Lemma \ref{tullelemma}, we know that 
    \begin{align}\lambda_{\max}^{s_0}(\Sigma_1 - \Sigma_2) \leq \lambda_{\max}^{s}(\Sigma_1 - \Sigma_2) \leq 4\lambda_{\max}^{s_0}(\Sigma_1 - \Sigma_2).
    \end{align}
    Since also $t_0/2 \leq t \leq t_0$, it holds that
    \begin{align}
t \left ( \frac{\lambda_{\max}^s ({\Sigma_1 - \Sigma_2})}{\lambda_{\max}^s(\Sigma_1) \vee \lambda_{\max}^s(\Sigma_2)}\right)^2  &\geq \frac{t}{16}  \left ( \frac{\lambda_{\max}^{s_0} ({\Sigma_1 - \Sigma_2})}{\lambda_{\max}^{s_0}(\Sigma_1) \vee \lambda_{\max}^{s_0}(\Sigma_2)} \right)^2 \\
&\geq \frac{t_0}{32} \left ( \frac{\lambda_{\max}^{s_0} ({\Sigma_1 - \Sigma_2})}{\lambda_{\max}^{s_0}(\Sigma_1) \vee \lambda_{\max}^{s_0}(\Sigma_2)} \right)^2\\
&\geq \frac{C}{32} \left\{ s_0 \log(ep/s_0) \vee \log\log(8n)\right\}\\
&\geq \frac{C}{32} \left\{ s \log(ep/s) \vee \log\log(8n)\right\}.   \label{prop5prooftemp}
\end{align}
    Since $C\geq 32$, the chain of inequalities preceding Equation \eqref{prop5prooftemp} imply that $t \geq s\log(ep/s)\vee \log\log(8n)$, since $\lambda_{\max}^s(\Sigma_1 - \Sigma_2) \leq \lambda_{\max}^s(\Sigma_1) \vee \lambda_{\max}^s(\Sigma_2)$, due to Lemma \ref{tullelemma}. It follows that $r(p,n,s,t) = t^{-1/2}\{ s \log(ep/s) \vee \log \log (8n)\}^{1/2}$ and thus
     \begin{align}
         \lambda_{\max}^s ({\Sigma_1 - \Sigma_2}) &\geq \sqrt{\frac{C}{32}} \left\{ \lambda_{\max}^s(\Sigma_1) \vee \lambda_{\max}^s(\Sigma_2)\right\} r(p,n,s,t).
    \end{align}

    On the event $\mathcal{E}$, we thus have
    \begin{align}
        S_{t,s} - \lambda \widehat{\sigma}^2_s r(p,n,s,t) &\geq \lambda_{\max}^s ({\Sigma_1 - \Sigma_2}) - c_5 \left\{\lambda_{\max}^s(\Sigma_1) \vee \lambda_{\max}^s(\Sigma_2)\right\} r(p,n,s,t)  \\
        & - \lambda  c_3 \left\{\lambda_{\max}^s(\Sigma_1) \wedge \lambda_{\max}^s(\Sigma_2)\right\}  r(p,n,s,t) \\
        &\geq \left( \sqrt{\frac{C}{32}} - c_5 - \lambda c_3 \right) \left\{  \lambda_{\max}^s(\Sigma_1) \vee \lambda_{\max}^s(\Sigma_2) \right\} r(p,n,s,t) \\
        &> 0,
    \end{align}
    since $C > 32 (\lambda c_3 + c_5)^2$. It follows that 
    \begin{align}
        \EE \left(1 - \psi_{\mathrm{adaptive},\lambda}(X)\right) &\leq \PP(\mathcal{E}^{\complement})\\
        &\leq \frac{\delta}{2},
    \end{align}
    and the proof is complete.

\end{proof}

\subsection{Proofs of results from Section \ref{adaptivecompsec}}
\begin{proof}[Proof of Proposition \ref{comuptationaltestresult}]
The proof follows the same strategy as the proof of Proposition \ref{adaptivemultivariateupperbound}. Fix any $\delta>0$, and assume that the observations $X_1, \ldots, X_n$ satisfy Assumption \ref{assmultivariate2} for some $w,u>0$. Write $X = (X_1^\top, \ldots, X_n^\top)^\top$ for the random vector in $\RR^{pn}$ consisting of the observed data. 
    
We first control the Type I error under the null hypothesis. Assume that $\Cov(X) = \mathrm{Diag}( \left\{\Sigma\right\}_{i\in[n]})$ for some  positive definite covariance matrix $\Sigma \in \RR^{p \times p}$. 
Define the events
    \begin{align}
    \mathcal{E}_9 &=  \left\{ c_4 \leq \frac{\widehat{\sigma}^2_{ \mathrm{con}}}{\lambda_{\max}^1({\Sigma})} \right\},\\
    \mathcal{E}_{10} &=  \bigcap_{t \in \mathcal{T}} \bigcap_{s \in \mathcal{S}} \left\{ \widehat{\lambda}_{\max}^s ( {\widehat{\Sigma}_{1,t} - \widehat{\Sigma}_{2,t}} ) \leq c_6 \lambda_{\max}^1({\Sigma}) h(p,n,s,t) \right\},
    \end{align}
    where $\widehat{\sigma}^2_{\mathrm{con}}$ is given in \eqref{convsigmahat}, $\widehat{\Sigma}_{i,t}$ is given in \eqref{sigmahats},  $\mathcal{T}$ is given in \eqref{mathcalt}, $\mathcal{S}$ is given in \eqref{mathcals}, $h(p,n,s,t)$ is given in \eqref{hdef},  $c_4 =  \delta^2 (\delta^2 + 16\delta + 64)^{-1}(2e\pi w^2)^{-1}$,  $c_6 = c_0\{1 + \log(4/\delta)\}$, and $c_0$ is the constant from Lemma \ref{covarianceconcentrationlemma} depending only on $u>0$.   Define the event $\mathcal{E} = \mathcal{E}_9 \cap \mathcal{E}_{10}$. By Lemma \ref{eventbound5} we have $\PP(\mathcal{E})\geq 1 - \delta/2$.

    On $\mathcal{E}$, we have
    \begin{align}
        \frac{\widehat{S}_{t,s}}{\widehat{\sigma}^2_{\mathrm{con}}} &\leq \frac{c_6 \lambda_{\max}^1 ({\Sigma}) h(p,n,s,t)}{c_4\lambda_{\max}^1 ({\Sigma})}\\
        &= \frac{c_6}{c_4}h(p,n,s,t),
    \end{align}
    for all $t\in \mathcal{T}$ and $s \in \mathcal{S}$. Choosing $\lambda \geq \lambda_0=c_6/c_4$, we thus obtain 
    \begin{align}
        \EE \widehat{\phi}_{\mathrm{adaptive},\lambda}(X) &\leq \PP(\mathcal{E}^\complement)\\
        &\leq \frac{\delta}{2}. 
    \end{align}
     Next we consider the Type II error. Let $\lambda$ remain unchanged, and assume now that $\Cov(X_i) = \Sigma_1$ for $i\leq t_0$ and $\Cov(X_i) = \Sigma_2$ for $i>t_0$, where $t_0 \in [n-1]$ and $\Sigma_1, \Sigma_2$ are some positive definite matrices in $\RR^{p\times p}$. 
     Assume further that $\Sigma_1, \Sigma_2$ and $t_0$ satisfy Equation \eqref{prop6condition} for some $s_0\in [p]$, where the constant $C$ satisfies
    \begin{align}
        C > 32 (\lambda c_3 + c_6)^2,
    \end{align}
    where $c_3 = 1 + c_0 \log(1+4/\delta)$ and $c_6$ is as before. 
    
    Without loss of generality, we may assume that $\min(t_0, n-t_0) = t_0$.  Note that there exists some $t \in \mathcal{T}$ such that $t_0/2 \leq t \leq t_0$ and some $s \in \mathcal{S}$ such that $s_0/2\leq s\leq s_0$. Fixing this $s$ and $t$, we define the events 
    \begin{align}
    \mathcal{E}_{11} &=  \left\{\frac{\widehat{\sigma}^2_{\mathrm{con}}}{\lambda_{\max}^1(\Sigma_1) \wedge \lambda_{\max}^1(\Sigma_2)} \leq  c_3\right\},\\
    \mathcal{E}_{12} &=  \bigcap_{i=1,2} \left\{ \widehat{\lambda}_{\max}^s ( {\widehat{\Sigma}_{i,t} - \Sigma_i}) \leq \frac{c_6}{2} \lambda_{\max}^1({\Sigma_i}) h(p,n,s,t) \right\},
    \end{align}
    where $c_3$ and $c_6$ are as before. Define $\mathcal{E} = \mathcal{E}_{11} \cap \mathcal{E}_{12}$. Due to Lemma \ref{eventbound6}, we have that $\PP(\mathcal{E})\geq 1 - \delta/2$. On $\mathcal{E}$, we have
    \begin{align}
          \widehat{\lambda}_{\max}^s ({\widehat{\Sigma}_{1,t} - \widehat{\Sigma}_{2,t}}) &\geq \widehat{\lambda}_{\max}^s({\Sigma_1 - \Sigma_2})- 2 \  \underset{i = 1,2}{\max} \  \widehat{\lambda}_{\max}^s({\widehat{\Sigma}_{i,t} - \Sigma_i})\\
          &\geq \lambda_{\max}^s({\Sigma_1 - \Sigma_2}) - c_6 \left\{ \lambda_{\max}^1(\Sigma_1) \vee \lambda_{\max}^1(\Sigma_2)\right\} h(p,n,s,t).
    \end{align}
    Due to the definition of $s_0/2 \leq s \leq s_0$ and Lemma \ref{tullelemma}, we know that 
    \begin{align}\lambda_{\max}^{s_0}(\Sigma_1 - \Sigma_2) \leq \lambda_{\max}^{s}(\Sigma_1 - \Sigma_2) \leq 4\lambda_{\max}^{s_0}(\Sigma_1 - \Sigma_2).
    \end{align}
    Since also $t_0/2 \leq t \leq t_0$, it holds that
    \begin{align}
t \left ( \frac{\lambda_{\max}^s ({\Sigma_1 - \Sigma_2})}{\lambda_{\max}^1(\Sigma_1) \vee \lambda_{\max}^1(\Sigma_2)}\right)^2  &\geq \frac{t}{16}  \left ( \frac{\lambda_{\max}^{s_0} ({\Sigma_1 - \Sigma_2})}{\lambda_{\max}^{s_0}(\Sigma_1) \vee \lambda_{\max}^{s_0}(\Sigma_2)} \right)^2 \\
&\geq \frac{t_0}{32} \left ( \frac{\lambda_{\max}^{s_0} ({\Sigma_1 - \Sigma_2})}{\lambda_{\max}^{s_0}(\Sigma_1) \vee \lambda_{\max}^{s_0}(\Sigma_2)} \right)^2\\
&\geq \frac{C}{32} s_0^2 \left\{ \log(ep) \vee \log\log(8n)\right\}\\
&\geq \frac{C}{32} s^2 \left\{ \log(ep) \vee \log\log(8n)\right\}.   \label{prop6prooftemp}
\end{align}
    Since $C\geq 32$, the chain of inequalities preceding Equation \eqref{prop6prooftemp} imply that \\$t \geq \left\{ \log(ep) \vee \log\log(8n)\right\}$, since $\lambda_{\max}^s(\Sigma_1 - \Sigma_2)$ can never exceed the maximum of $\lambda_{\max}^s(\Sigma_1)$ and $\lambda_{\max}^s(\Sigma_2)$. It follows that $h(p,n,s,t) = s t^{-1/2} \{ \log (ep) \vee \log \log (8n)\}^{1/2}$, and thus
     \begin{align}
         \lambda_{\max}^s ({\Sigma_1 - \Sigma_2}) &\geq \sqrt{\frac{C}{32}} \left\{ \lambda_{\max}^1(\Sigma_1) \vee \lambda_{\max}^1(\Sigma_2)\right\}  h(p,n,s,t).
    \end{align}

    On the event $\mathcal{E}$, we thus have
    \begin{align}
        & \widehat{S}_{t,s} - \lambda \widehat{\sigma}^2_{\mathrm{con}} h(p,n,s,t) \\
        \geq & \lambda_{\max}^s ({\Sigma_1 - \Sigma_2}) - c_6 \left\{\lambda_{\max}^1(\Sigma_1) \vee \lambda_{\max}^1(\Sigma_2)\right\} h(p,n,s,t)  \\
         - & \lambda  c_3 \left\{\lambda_{\max}^1(\Sigma_1) \vee \lambda_{\max}^1(\Sigma_2)\right\}  h(p,n,s,t) \\
        \geq  & \left( \sqrt{\frac{C}{32}} - c_6 - \lambda c_3 \right) \left\{  \lambda_{\max}^1(\Sigma_1) \vee \lambda_{\max}^1(\Sigma_2) \right\} h(p,n,s,t) \\
        > & 0,
    \end{align}
    since $C > 32 (\lambda c_3 + c_6)^2$. It follows that 
    \begin{align}
        \EE \left(1 - \widehat{\psi}_{\mathrm{adaptive},\lambda}(X)\right) &\leq \PP(\mathcal{E}^{\complement})\\
        &\leq \frac{\delta}{2},
    \end{align}
    and the proof is complete.

\end{proof}

\section{Auxiliary Lemmas}
\begin{lemma}\label{lowerchernoff}
Let $X_1, \ldots, X_n$ be independent random variables, and assume that each $X_i /\sigma$ has a continuous density bounded above by $w$ for $i=1, \ldots, n\in \NN$ and some $w>0$. Let $S = \sum_{i=1}^n X_i^2$. Then for any $x>0$ we have
\begin{align}
    \PP\left(  S \leq \sigma^2 x\right) &\leq \exp \left[ \frac{n}{2} \left\{1 + \log\left(2\pi w^2\right) - \log\left(\frac{n}{x}\right) \right\} \right].
\end{align}
\end{lemma}
\begin{proof}
    By a Chernoff bound, we have
    \begin{align}
        \PP\left(  S \leq \sigma^2 x\right) &\leq \underset{\lambda<0}{\inf} \ \exp\left( -\lambda \sigma^2 x\right) \prod_{i=1}^n M_{Y_i} \left( \lambda \right) ,
    \end{align}
    where $M_{Y_i}$ denotes the Moment Generating Function of $Y_i = X_i^2$. Now, the density $f_{Y_i}$ of $Y_i$ satisfies $f_{Y_i}(y) \leq y^{-1/2} w /\sigma$, which implies that
    \begin{align}
        M_{Y_i}(\lambda) &\leq \int_{y=0}^{\infty} \frac{\exp\left( \lambda y \right) y^{-1/2} w}{\sigma}dy\\
        &=\frac{1}{\sqrt{-\lambda}} \frac{\sqrt{\pi}w}{\sigma},
    \end{align}
    for $\lambda<0$. Hence, 
    \begin{align}
        \PP\left(  S \leq \sigma^2 x\right) &\leq \underset{\lambda<0}{\inf} \ \exp\left\{\frac{n}{2} \log \left( - \frac{\pi w^2}{\lambda \sigma^2} \right) -\lambda \sigma^2 x\right\}.
    \end{align}
    Set $\lambda = - n \left(2\sigma^2x\right)^{-1}$ to obtain the desired result. 
\end{proof}

The following Lemma can be viewed as a unification of Proposition 4.2 in \cite{berthet_optimal} and Theorem 1 in \cite{koltchinskii2017}.
\begin{lemma}\label{covarianceconcentrationlemma}
    Fix any $p\in \NN$ and $s\in [p]$, and let $X_i$ be centered and independent $p$-dimensional sub-Gaussian random variables with $\EE X_i X_i^\top = \Sigma$, for $i=1, \ldots, n$ and some $\Sigma \in \RR^{p \times p}$.  Assume further that $\normmo{X_i}^2 \leq u \normmop{\Sigma}$ for all $i$ and some $u>0$. Let $\widehat{\Sigma} = n^{-1} \sum_{i=1}^n X_i X_i^\top$ and let $\lambda_{\max}^s(\cdot)$ be defined as in \eqref{firsteigsparse}. There exists a constant $c_0>0$ depending only on $u$, such that, for all $x\geq 1$, we have
    \begin{align}
        \PP \left[\lambda_{\max}^s (\widehat{\Sigma}- \Sigma) \geq c_0 \normm{\Sigma}_{\mathrm{op}} \left\{   \sqrt{\frac{s\log(ep/s)}{n}}\vee \frac{s\log(ep/s)}{n} \vee \sqrt{\frac{x}{n}} \vee \frac{x}{n}  \right\}  \right] \leq e^{-x}. \label{eqberthet}
    \end{align}
    Moreover, if $\normmo{v^\top X_i}^2 \leq u ( v^\top \Sigma v)$ for any $v \in S^{p-1}$, the factor $\normm{\Sigma}_{\mathrm{op}}$ on the left hand side of \eqref{eqberthet} can be replaced by $\lambda_{\max}^s ({\Sigma})$. 
\end{lemma}
\begin{proof}
    The proof follows a similar strategy as the proof of Proposition 4.2 in \cite{berthet_optimal}, which is a slightly less general result. We begin by claiming that there exist a subset $\mathcal{N}$ the unit sphere $S^{p-1}$ with cardinality at most ${p \choose s} 9^s$ such that for any symmetric $A\in \RR^{p\times p}$, we have
    \begin{align}
        \lambda_{\max}^s (A) &\leq 2 \ \underset{v \in \mathcal{N}}{\max} \  |v^\top A v|\label{claim1}. 
    \end{align}
    To show this, note first that 
    \begin{align}
        \lambda_{\max}^s (A)  &= \underset{\substack{I \subseteq [p]\\ |I| =s}}{\max} \ \lambda_{\max}^s(A_I),
    \end{align}
    where $A_I = \left\{A_{i,j}\right\}_{(i,j)\in I \times I} \in \RR^{s\times s}$ is the sub-matrix of $A$ in which the rows and columns not contained in the index set $I$ have been removed. Now, for any $I \subseteq [p]$ such that $|I| = s$, we have
    \begin{align}
        \lambda_{\max}(A_I) &= \underset{v \in S^{s-1}}{\sup} \ |v^\top A_I v|\\
        &= \normmop{A_I},
    \end{align}
    where the last inequality follows from $A$ being symmetric, in which case operator norm agrees with the largest absolute eigenvalue of $A$. As in the proof of Theorem 4.4.5 in \cite{Vershynin_2018}, as well as the results in Exercise 4.4.3 in the same book, there exists an $1/4$-net $\mathcal{N}_I$ of the unit sphere ${S}^{s-1}$ with cardinality at most $9^s$ such that 
    \begin{align}
        \normmop{A_I} &\leq 2 \underset{x \in \mathcal{N}_I}{\max} \ |x^\top A_I x|.
    \end{align}
    For any vector $v \in \RR^{p}$ and $I \subseteq [p]$, let $v_I$ denote the sub-vector of $v$ containing the $i$-th entry of $v$ if and only if $i \in I$. Now set
    \begin{align}
        \mathcal{N} = \bigcup_{I \subseteq [p], |I| = s} \left\{  v \in \RR^p \ ; \ v_I  \in \mathcal{N}_I, v_{I^\complement} = 0      \right\}, 
    \end{align}
    which has cardinality at most ${p \choose s} 9^s$, proving the claim. It follows that
    \begin{align}
        \lambda_{\max}^s (\widehat{\Sigma} - \Sigma) &\leq 2 \underset{v \in \mathcal{N}}{\max} \  |v^\top (\widehat{\Sigma} - \Sigma) v|.
    \end{align}
    Now, for any $v \in \mathcal{N}$ we have
    \begin{align}
        v^\top \left( \widehat{\Sigma}-\Sigma \right)v &= \frac{1}{n}\sum_{i=1}^n \left[\left(v^\top X_i\right)^2 - \EE \left\{ \left( v^\top X_i\right)^2\right\} \right]\label{simplification}.
    \end{align}

    Since the $X_i$ are sub-Gaussian, the summands on the right hand side of Equation \eqref{simplification} are independent and mean-zero sub-Exponential random variables\footnote{The Orlicz-1 norm $\normm{X}_{\Psi_1}$ of a real-valued random variable $X$ is defined as $\normm{X}_{\Psi_1} = \inf \ \{ \lambda \geq 0 \ ; \ \EE \exp (|X|/\lambda) \leq 2\}$, and $X$ is said to be sub-Exponential if $\normm{X}_{\Psi_1}\leq \infty$. }, with Orlicz-1 norm 
    \begin{align}
        \normmotwo{\left(v^\top X_i\right)^2} &= \normmo{v^\top X_i}^2\\
        &\leq \normmo{X_i}^2\\
        &\leq u \normmop{\Sigma},\label{orliczstuff}
    \end{align}
    where the first equality is due to Lemma 2.7.6 in \cite{Vershynin_2018}. Since centering of a sub-Gamma random variable preserves the Orlicz-1 norm up to a universial constant, it follows that $\lVert v^\top (\widehat{\Sigma} - \Sigma)v \rVert_{\Psi_1} \lesssim u \normmop{\Sigma}$. From Bernstein's Inequality (Theorem 2.8.1 in \citeauthor{Vershynin_2018} \citeyear{Vershynin_2018}), we thus have
    \begin{align}
        \PP \left[  \left| \frac{1}{n}\sum_{i=1}^n \left[\left(v^\top X_i\right)^2 - \EE \left\{ \left( v^\top X_i\right)^2\right\} \right] \right|  \geq  c u \normmop{\Sigma} \left( \sqrt{\frac{t}{n}} \vee \frac{t}{n}\right) \right] &\leq 2e^{-t},
    \end{align}
    for some absolute constant $c\geq 1$ and any $t>0$. By a union bound over all $v \in \mathcal{N}$, we obtain 
    \begin{align}
        \PP \left\{ \lambda_{\max}^s ( {\widehat{\Sigma}- \Sigma}) \geq {2c u \normmop{\Sigma}} \left( \sqrt{\frac{t}{n}} \vee \frac{t}{n}\right)  \right\} &\leq 2 {p \choose s} 9^s e^{-t}\\
        &\leq 2 \exp\left\{ 4 s \log\left(\frac{ep}{s}\right) - t\right\}.
    \end{align}
    Now set $t = 8s\log(ep/s) \vee 2(1+\log 2)x $ to obtain the first claim of the Lemma,  letting $c_0 = 16cu$. 

    For the second claim, note that if $\normmo{v^\top X_i}^2 \leq u ( v^\top \Sigma v)$ for any $v \in S^{p-1}$, then then the right hand side of the inequality in \eqref{orliczstuff} may be replaced by $u (v^\top \Sigma v)$, which is bounded above by $\lambda_{\max}^s(\Sigma)$ for any $v \in S^{p-1}_s$. The second claim then holds with the same constant $c_0$ as before. 
\end{proof}

\begin{lemma}\label{eventbound1}
Let $X_1, \ldots, X_n$ be univariate random variables satisfying Assumption \ref{assunivariate} for some $u>0$, and assume that $\Var(X_i) = \sigma^2>0$ for all $i\in [n]$. For any $t$, let $\widehat{\sigma}^2_{1,t} = \frac{1}{t}\sum_{i=1}^t X_i^2$ and $\widehat{\sigma}^2_{2,t} = \frac{1}{t} \sum_{i=1}^t X_{n-i+1}^2$. Let $\mathcal{T}$ be as in \eqref{mathcalt}, and fix any $\delta \in (0,1)$. Define the events
\begin{align}
    \mathcal{E}_1 &= \bigcap_{i=1,2} \bigcap_{ t \in \mathcal{T}} \left\{ \widehat{\sigma}_{i,t}^2 \geq \sigma^2 c_1  \right\},\\
    \mathcal{E}_2 &= \bigcap_{i=1,2} \bigcap_{ t \in \mathcal{T}} \left\{  \left| \widehat{\sigma}_{i,t}^2 - \sigma^2\right| \leq \sigma^2  c_2 \left( \frac{\log\log(8n)}{t} \vee \sqrt{\frac{\log\log(8n)}{t}}\right)  \right\},
\end{align}
where $c_1 = \delta^2 (16 w)^{-2} (2e\pi)^{-1}$, $c_2 = c_0 \left\{ 2 + \log\left(8/\delta\right)\right\} \vee 1$ and $c_0$ is the constant from Lemma \ref{covarianceconcentrationlemma} depending only on $u$. Then we have $\PP(\mathcal{E}_1 \cap \mathcal{E}_2) \geq 1- \delta/2$.
\end{lemma}
\begin{proof}
    Observe that by Lemma \ref{lowerchernoff}, for  we have 
\begin{align}
    \PP\left( \mathcal{E}_1^\complement\right) &\leq \sum_{i=1,2}\sum_{t \in \mathcal{T}} \PP  \left( \widehat{\sigma}_{i,t}^2 \leq \sigma^2 c_1   \right)\\
    &\leq 2\sum_{k=0}^{\infty} \left(\frac{\delta}{16}\right)^{2^k}\\
    &\leq \frac{\delta}{4}. 
\end{align}
Moreover, using Lemma \ref{covarianceconcentrationlemma}, we have
\begin{align}
    \PP \left( \mathcal{E}_2^\complement \right) &\leq \sum_{i=1,2}\sum_{t \in \mathcal{T}} \PP\left\{   \left| \widehat{\sigma}_{i,t}^2  - \sigma^2\right| \geq \sigma^2 c_2 \left( \frac{\log\log(8n)}{t} \vee \sqrt{\frac{\log\log(8n)}{t}}\right)  \right\}\\
    &\leq 2 \left| \mathcal{T}\right| \exp\left\{ - (c_2/c_0)  \log\log(8n)\right\}\\
    &\leq 2\frac{\log_2(n/2)+1}{\log(8n)^{c_2/c_0}}\\
    &\leq \frac{\delta}{4},
\end{align}
where we in the last inequality used the $c_2$ and the fact that \begin{align}
\{\log \log (8n)\}^{-1} \log \{ 1 + \lfloor \log_2(n/2)\rfloor\} \leq 2
\end{align} and $\log\log(8n)\geq 1$ for all $n \geq 2$.
\end{proof}

\begin{lemma}\label{eventbound2}
For some $t \in \NN$, let $X_1, \ldots, X_t$ and $Y_1, \ldots, Y_t$ be two independent sets of univariate random variables, each satisfying Assumption \ref{assunivariate} for some $w,u>0$. Assume that $\Var(X_i) = \sigma_1^2$ and $\Var(Y_i) = \sigma_2^2$ for all $i\in [t]$ and some $\sigma_1^2, \sigma_2^2>0$. 
Let $\widehat{\sigma}^2_{1,t} = \frac{1}{t}\sum_{i=1}^t X_i^2$ and $\widehat{\sigma}^2_{2,t} = \frac{1}{t} \sum_{i=1}^t Y_{i}^2$. Fix any $\delta\in(0,1)$ and define the events
\begin{align}
    \mathcal{E}_3 &= \left\{   \widehat{\sigma}_{2,t}^2 \geq  \sigma_2^2 \left( c_1 \vee  \left[1 -c_0\left\{ \frac{\log(8/\delta)}{t} \vee \sqrt{\frac{\log(8/\delta)}{t}}\right\}  \right]   \right)  \right\}\\
    \mathcal{E}_4 &= \left\{ \widehat{\sigma}_{1,t}^2 \leq \sigma_1^2 \left[ 1 + c_0\left\{ \frac{\log(8/\delta)}{t} \vee \sqrt{\frac{\log(8/\delta)}{t}}\right\} \right] \right\}, 
\end{align}
where $c_1 = \delta^2 (16 w)^{-2} (2e\pi)^{-1}$ and $c_0$ is the constant from Lemma \ref{covarianceconcentrationlemma} depending only on $u$. Then $\PP\left( \mathcal{E}_3 \cap \mathcal{E}_4\right)\geq 1 - \delta /2$.
\end{lemma}
\begin{proof}
     Similar to the proof of Lemma \ref{eventbound1}, we have
\begin{align}
    \PP\left(   \widehat{\sigma}_{2,t}^2 \leq  c_1 \sigma_2^2 \right) &\leq \frac{\delta}{4}, 
\end{align}
using Lemma \ref{lowerchernoff}. By Lemma \ref{covarianceconcentrationlemma}, we also have
    \begin{align}
        \PP \left(   \widehat{\sigma}_{2,t}^2 \leq  \sigma_2^2  \left[1 -c_0\left\{ \frac{\log(8/\delta)}{t} \vee \sqrt{\frac{\log(8/\delta)}{t}}\right\}  \right]  \right) &\leq \frac{\delta}{8}, 
        \intertext{and}
        \PP \left(    \widehat{\sigma}_{1,t}^2 \geq \sigma_1^2 \left[ 1 + c_0 \left\{ \frac{\log(8/\delta)}{t} \vee \sqrt{\frac{\log(8/\delta)}{t}}\right\} \right]  \right)  &\leq \frac{\delta}{8}.
    \end{align}
    It follows that $\PP\left(\mathcal{E}_3 \cap \mathcal{E}_4\right) \geq 1 - \PP\left( \mathcal{E}_3^\complement\right) - \PP\left( \mathcal{E}_4^\complement\right) \geq 1 - \delta/2$. 
    
\end{proof}

\begin{lemma}\label{eventbound3}
Let $X_1, \ldots, X_n$ be $p$-dimensional random variables satisfying Assumption \ref{assmultivariate2} for some $w,u>0$. Assume further that $\EE X_i X_i^\top = \Sigma$ for all $i\in[n]$ and some positive definite matrix $\Sigma \in \RR^{p\times p}$.  Let $\widehat{\Sigma}_{1,t} = t^{-1} \sum_{i=1}^t X_i X_i^\top$ and $\widehat{\Sigma}_{2,t} = t^{-1} \sum_{i=n-t+1}^n X_i X_i^\top$. Define $\gamma(s) = \gamma(p,n,s)=  s\log(ep/s)\vee \log\log(8n)$ and for any $s \in [p]$ such that $\gamma(s)\leq n$, let 
\begin{align}
    \widehat{\sigma}^2_s = \lambda_{\max}^s(\widehat{\Sigma}_{1, \lceil \gamma(s) \rceil})  \wedge \lambda_{\max}^s(\widehat{\Sigma}_{2, \lceil \gamma(s) \rceil}),
\end{align}
where $\lambda_{\max}^s(\cdot)$ is defined in \eqref{firsteigsparse}.  
For any fixed $\delta\in (0,1)$, define the events
    \begin{align}
    \mathcal{E}_5 &= \bigcap_{\substack{s \in \mathcal{S};\\ \gamma(s)\leq n}}\left\{ c_4 \leq \frac{\widehat{\sigma}^2_s}{\lambda_{\max}^s(\Sigma) }\right\},\\
    \mathcal{E}_6 &=  \bigcap_{t \in \mathcal{T}} \bigcap_{s \in \mathcal{S}} \left\{ \lambda_{\max}^s (\widehat{\Sigma}_{1,t} - \widehat{\Sigma}_{2,t})  \leq c_5 \lambda_{\max}^s(\Sigma) r(p,n,s,t) \right\},
    \end{align}
    where $\mathcal{S}$ is given in \eqref{mathcals}, $\mathcal{T}$ is given in \eqref{mathcalt},  $r(p,n,s,t)$ is given in \eqref{rdef}, $c_4 =  \delta^2 (\delta^2 + 16\delta + 64)^{-1}(2e\pi w^2)^{-1}$, $c_5 = 4c_0 \log(1 + 16/\delta)$, and $c_0$ is the constant from Lemma \ref{covarianceconcentrationlemma} depending only on $u>0$. Then $\PP\left( \mathcal{E}_5 \cap \mathcal{E}_6 \right)\geq 1- \delta/2$.
\end{lemma}
\begin{proof}
    We first show that  $\PP(\mathcal{E}_5^\complement)\leq \delta/4$. Fix any $s\in \mathcal{S}$, and let $v \in S^{p-1}_s$ be an $s$-sparse vector satisfying $v^\top \Sigma v = \lambda_{\max}^s(\Sigma)$.\footnote{Note that such a $v$ always exists, since $S^{p-1}_s = \{v \in S^{p-1} \ ; \ \normm{v}_0\leq s\}$ is a closed (and thus compact) subset of $S^{p-1}$ for any $s \in [p]$. To see this, consider any sequence $\{v_i\}$ in $S^{p-1}_s$ converging to some $v \in S^{p-1}$. Suppose for contradiction that $\normm{v}_0 >s$, and let $0<\epsilon < \min\{ |v(j) : v(j)\neq 0\}$. There exists an $i$ such that $\normm{v_i - v}_2 < \epsilon$. In particular, since $v_i \in S^{p-1}_s$ and $\normm{v}_0>s$ it must be that $v_{i}(j)=0$ and $v(j) \neq 0$ for some $j$. For this $j$, we have $|v(j)| = |v(j) - v_{i}(j)| < \epsilon$, which is a contradiction. }  By Lemma \ref{lowerchernoff} and a union bound, we have that
    \begin{align}   
    \PP \left(\mathcal{E}_5^\complement\right) &= \PP \left(  \bigcup_{\substack{s \in \mathcal{S};\\ \gamma(s)\leq n}}\left\{ \frac{\widehat{\sigma}^2_s}{\lambda_{\max}^s(\Sigma)} \leq c_4 \right\}\right)\\
    &\leq 2 \sum_{\substack{s \in \mathcal{S};\\ \gamma(s) \leq n}} \PP \left(   \lambda_{\max}^{s}\left(  \lceil \gamma(s) \rceil^{-1}\sum_{i=1}^{\lceil \gamma(s) \rceil} X_i X_i^\top  \right)  \leq  c_4 \lambda_{\max}^s(\Sigma) \right)\\
     &\leq 2 \sum_{\substack{s \in \mathcal{S};\\ \gamma(s) \leq n}} \PP \left(   \sum_{i=1}^{\lceil \gamma(s) \rceil} (v^\top X_i)^2    \leq  \lceil \gamma(s) \rceil c_4 v^\top \Sigma v \right)\\
    &\leq  2 \sum_{s\in \mathcal{S}} \exp \left[ \frac{\lceil \gamma(s) \rceil}{2} \left\{1 + \log\left(2\pi w^2\right) - \log\left(\frac{1}{c_4}\right) \right\} \right]\\
    &\leq  2 \sum_{s\in \mathcal{S}} \exp \left\{  - \frac{1}{2}{s\log\left( \frac{ep}{s}\right)}  \log\left(\frac{1}{2e c_4\pi w^2}\right)  \right\}\\
    &\leq  2 \sum_{s\in \mathcal{S}} \exp \left\{  - \frac{1}{2}{s}  \log\left(\frac{1}{2e c_4\pi w^2}\right)  \right\}\\
    &\leq 2 \sum_{s=1}^{\infty} (2ec_4 \pi w^2)^{s/2}\\
    &\leq \frac{2}{(2ec_4 \pi w^2)^{-1/2}-1}\\
    &= \delta / 4,
\end{align}
where we in the third inequality used Assumption \ref{assmultivariate2-b}, and in the last equality used the definition of $c_4$. It follows that $\PP(\mathcal{E}_5) \geq 1-\delta/4$. 

Now consider $\mathcal{E}_6$.  Lemma \ref{covarianceconcentrationlemma} implies that
    \begin{align}
        \PP\left (\mathcal{E}_6^\complement \right ) &\leq \sum_{t \in \mathcal{T}} \sum_{s \in \mathcal{S}} \PP \left\{ \lambda_{\max}^s \left( {\widehat{\Sigma}_{1,t} - \widehat{\Sigma}_{2,t}} \right) \geq c_5 \lambda_{\max}^s (\Sigma) r(p,n,s,t) \right\}\\
         &\leq 2 \sum_{t \in \mathcal{T}} \sum_{s \in \mathcal{S}} \PP \left\{ \lambda_{\max}^s \left( {\widehat{\Sigma}_{1,t} - \Sigma} \right) \geq \frac{c_5}{2}\lambda_{\max}^s(\Sigma) r(p,n,s,t) \right\}\\
         &\leq 2 \sum_{t \in \mathcal{T}} \sum_{s \in \mathcal{S}}  \exp \left[  - \frac{c_5}{2c_0 } \left\{ s \log (ep/s) \vee \log\log(8n) \right\}  \right]\\
         &\leq 2  \sum_{t \in \mathcal{T}} \sum_{s \in \mathcal{S}}  \exp \left[  - \frac{c_5}{4c_0 } \left\{ s \log (ep/s) + \log\log(8n) \right\}  \right]\\
        &\leq 2\left|\mathcal{T}\right|  \exp \left\{   -\frac{c_5}{4c_0} \log\log(8n)  \right\} \sum_{s \in \mathcal{S}} \exp\left \{-\frac{c_5}{4c_0} s \log\left(\frac{ep}{s}\right)\right\}\\
        &\leq 2 \frac{1 + \lfloor \log_2(n/2) \rfloor}{\log(8n)^{c_5 / (4c_0)}} \sum_{s=1}^{\infty} \exp\left(  -\frac{c_5}{4c_0}s   \right)\\
        &\leq \underset{n \in \NN  \ ; \ n\geq 2}{\sup} \frac{ 1 + \lfloor \log_2(n/2) \rfloor}{\log(8n)} \frac{2}{\exp\left\{ \frac{c_5}{4c_0}\right\}-1}\\
        &\leq  \frac{4}{\exp\left\{ \frac{c_5}{4c_0}\right\}-1}\\
        &=\delta/4,
    \end{align}
    where we used that $c_5 / (4c_0) \geq 1$ in the third and second last inequalities and the definition of $c_5$ in the last. The proof is complete. 
    
\end{proof}

\begin{lemma}\label{eventbound4}
Let $X_1, \ldots, X_n$ be $p$-dimensional random variables satisfying Assumption \ref{assmultivariate2} for some $w,u>0$. For some $t \leq \lfloor n/2 \rfloor$, assume that $\EE X_i X_i^\top = \Sigma_1$ for $i \leq t$ , and $\EE X_i X_i^\top = \Sigma_2$ for $i \geq n - t+1$, where $\Sigma_1$ and $\Sigma_2$ are two positive definite matrices in $\RR^{p \times p}$.  Let $\widehat{\Sigma}_{1,t} = t^{-1} \sum_{i=1}^t X_i X_i^\top$ and $\widehat{\Sigma}_{2,t} = t^{-1} \sum_{i=n-t+1}^n X_i X_i^\top$, and let $\gamma(s) = \gamma(p,n,s) =  s\log(ep/s)\vee \log\log(8n)$. For any fixed $s \in [p]$ such that $\gamma(s) \leq n$ define
\begin{align}
    \widehat{\sigma}^2_s = \lambda_{\max}^s(\widehat{\Sigma}_{1, \lceil \gamma(s) \rceil})  \wedge \lambda_{\max}^s(\widehat{\Sigma}_{2,\lceil \gamma(s) \rceil}),
\end{align}
where $\lambda_{\max}^s(\cdot)$ is defined in \eqref{firsteigsparse}. 
For the same $s$ and any $\delta\in (0,1)$, define the events
    \begin{align}
    \mathcal{E}_7 &=  \left\{\frac{\widehat{\sigma}^2_s}{\lambda_{\max}^s (\Sigma_1) \wedge \lambda_{\max}^s(\Sigma_2) } \leq c_3\right\},\\
    \mathcal{E}_8 &=  \bigcap_{i=1,2} \left\{ \lambda_{\max}^s ({\widehat{\Sigma}_{i,t} - \Sigma_i} ) \leq \frac{c_5}{2} \lambda_{\max}^s(\Sigma_i) r(p,n,s,t) \right\},
    \end{align}
    where $r(p,n,s,t)$ is given in \eqref{rdef}, $c_3 = 1 + c_0 \log(1+4/\delta)$, and $c_5 = 4c_0 \log(1 + 16/\delta)$, and $c_0$ is the constant from Lemma \ref{covarianceconcentrationlemma} depending only on $u>0$. Then $\PP\left( \mathcal{E}_7 \cap \mathcal{E}_8 \right)\geq 1- \delta/2$.
\end{lemma}
\begin{proof}
We first show that $\PP(\mathcal{E}_7^\complement)\leq \delta/4$. Without loss of generality, assume that $\lambda_{\max}^s(\Sigma_1) \leq  \lambda_{\max}^s(\Sigma_2)$. We have
    \begin{align}
    \PP\left(\mathcal{E}_7^\complement\right)& \leq  
    \PP \left\{ \frac{\lambda_{\max}^s (\widehat{\Sigma}_{1, \lceil \gamma(s) \rceil})}{\lambda_{\max}^s(\Sigma_1)} \geq c_3  \right\} \\
    &= \PP\left\{  \lambda_{\max}^s \left( \lceil \gamma(s) \rceil^{-1}\sum_{i=1}^{\lceil \gamma(s) \rceil} X_i X_i^\top \right) \geq c_3 \lambda_{\max}^s(\Sigma_1) \right\}\\
    &\leq \PP\left\{ \lambda_{\max}^s \left( \lceil \gamma(s) \rceil^{-1}\sum_{i=1}^{\lceil \gamma(s) \rceil} X_i X_i^\top  - \Sigma \right) \geq (c_3 -1) \lambda_{\max}^s(\Sigma_1) \right\}.
    \intertext{Now we apply Lemma \ref{covarianceconcentrationlemma} with $x = c_0^{-1}(c_3 -1)  \gamma(s)$ and $c_0$ the constant from that Lemma to obtain}
    \PP \left\{   \frac{\widehat{\sigma}^2_s}{\lambda_{\max}^s(\Sigma_1)} \geq c_3 \right\} &\leq \exp\left\{ -(c_3 - 1) c_0^{-1} s \log(ep/s)\right\}\\
    &\leq \frac{1}{\exp\left\{ (c_3 - 1) c_0^{-1} \right\}-1}\\
    &\leq \frac{\delta}{4},
\end{align}
here using the definition of $c_3$. It follows that $\PP(\mathcal{E}_7) \geq 1 - \delta/4$.

Now consider $\mathcal{E}_8$.  Similar to the proof of Lemma \ref{eventbound3}, we apply Lemma \ref{covarianceconcentrationlemma} to obtain 
    \begin{align}
    \PP(\mathcal{E}_8^\complement) 
         &\leq 2 \PP \left\{ \lambda_{\max}^s ( {\widehat{\Sigma}_{1,t} - \Sigma}) \geq \frac{c_5}{2}\lambda_{\max}^s(\Sigma_1) r(p,n,s,t) \right\}\\
        &\leq  \frac{4}{\exp\left\{ \frac{c_3}{4c_0}\right\}-1}\\
        &\leq \delta/4,
    \end{align}
    using the definition of $c_5$. The proof is complete. 
    
\end{proof}

\begin{lemma}\label{eventbound5}
Assume $n \geq \log(ep)$ and let $X_1, \ldots, X_n$ be $p$-dimensional random variables satisfying Assumption \ref{assmultivariate2} for some $w,u>0$. Assume further that $\EE X_i X_i^\top = \Sigma$ for all $i \in [n]$ and some positive definite matrix $\Sigma \in \RR^{p\times p}$. Let $\widehat{\Sigma}_{1,t} = t^{-1} \sum_{i=1}^t X_i X_i^\top$ and $\widehat{\Sigma}_{2,t} = t^{-1} \sum_{i=n-t+1}^n X_i X_i^\top$ and let
\begin{align}
    \widehat{\sigma}^2_{\mathrm{con}} &= \widehat{\lambda}^1_{\max} (\widehat{\Sigma}_{1,\lceil \log(ep)\rceil})  \wedge \widehat{\lambda}^1_{\max} (\widehat{\Sigma}_{2,\lceil \log(ep)\rceil}),
\end{align}
where $\widehat{\lambda}_{\max}^{s}(\cdot)$ is defined in \eqref{conveigdef}. 
For any $\delta\in (0,1)$, define the events
    \begin{align}
    \mathcal{E}_9 &=  \left\{ c_4 \leq \frac{\widehat{\sigma}^2_{\mathrm{con}}}{\lambda_{\max}^1({\Sigma})} \right\},\\
    \mathcal{E}_{10} &=  \bigcap_{t \in \mathcal{T}} \bigcap_{s \in \mathcal{S}} \left\{ \widehat{\lambda}_{\max}^s ( {\widehat{\Sigma}_{1,t} - \widehat{\Sigma}_{2,t}} ) \leq c_6 \lambda_{\max}^1({\Sigma}) h(p,n,s,t) \right\},
    \end{align}
    where $\mathcal{S}$ is given in \eqref{mathcals}, $\mathcal{T}$ is given in \eqref{mathcalt}, $\lambda_{\max}^s(\cdot)$ is defined in \eqref{firsteigsparse}, $h(p,n,s,t)$ is given in \eqref{hdef}, $c_4 =  \delta^2 (\delta^2 + 16\delta + 64)^{-1}(2e\pi w^2)^{-1}$,  $c_6 = c_0\{1 + \log(4/\delta) \}$, and $c_0$ is the constant from Lemma \ref{covarianceconcentrationlemma} depending only on $u>0$. Then $\PP\left( \mathcal{E}_9 \cap \mathcal{E}_{10}\right)\geq 1- \delta/2$.
\end{lemma}
\begin{proof}
    We first show that  $\PP(\mathcal{E}_9^\complement)\leq \delta/4$. Since $\widehat{\lambda}_{\max}^{s} (\cdot)$ is a convex relaxation of the implicit optimization problem giving $\lambda_{\max}^s(\cdot)$, we have that
    \begin{align}
        \widehat{\lambda}_{\max}^1 (A) \geq \lambda_{\max}^1 (A)
    \end{align}
    for any symmetric matrix $A$. By symmetry and the same arguments as in the proof of Lemma \ref{eventbound3} it follows that
    \begin{align}
        \PP(\mathcal{E}_9) &=  1 - 2 \PP \left\{     { \lambda^1_{\max} (\widehat{\Sigma}_{1,\lceil \log(ep)\rceil}) } < c_4 \lambda_{\max}^1({\Sigma}) \right\}  \\
        &\geq 1 - \delta/4.
    \end{align}

Now consider $\mathcal{E}_{10}$. Note that by a union bound, we have
\begin{align}
    \PP(\mathcal{E}_{10}^\complement) &\leq \sum_{t \in \mathcal{T}} \sum_{s \in \mathcal{S}} \PP \left \{   \widehat{\lambda}_{\max}^s ( {\widehat{\Sigma}_{1,t} - \widehat{\Sigma}_{2,t}} ) > c_6 \lambda_{\max}^1({\Sigma}) h(p,n,s,t)   \right\}
\end{align}
For any $t \in \mathcal{T}$ and $s \in \mathcal{S}$, Lemma \ref{duallemma} implies that
\begin{align}
    \widehat{\lambda}_{\max}^s ( {\widehat{\Sigma}_{1,t} - \widehat{\Sigma}_{2,t}} ) &\leq s\normm{{\widehat{\Sigma}_{1,t} - \widehat{\Sigma}_{2,t}}}_{\infty}\\
    &\leq  s\normm{{\widehat{\Sigma}_{1,t} - \Sigma_1}}_{\infty} + s\normm{{\widehat{\Sigma}_{2,t} - \Sigma_2}}_{\infty}
\end{align}
Now, let $Y_{i,t,j,k}$ denote the $(j,k)$-th element of $\widehat{\Sigma}_{i,t} - \Sigma_i$, for $i = 1,2$, and $j,k \in [p]$, so that e.g. 
\begin{align}
    Y_{1,t, j,k} = \frac{1}{t} \sum_{i=1}^t \left\{ X_i(j) X_i(k) - \EE X_i(j)X_i(k)\right\}. 
\end{align}
By symmetry, we have that
\begin{align}
    \PP(\mathcal{E}_{10}^\complement)  &\leq 2 \sum_{t \in \mathcal{T}} \sum_{s \in \mathcal{S}} \sum_{j \in [p]} \sum_{k \in [p]} \PP \left\{ | Y_{1,t,j,k}| > \frac{c_6}{2s} \lambda_{\max}^1(\Sigma) h(p,n,s,t)    \right\}. 
\end{align}
Due to Lemma 2.7.7 in \cite{Vershynin_2018} and Assumption \ref{assmultivariate2}, we have that
\begin{align}
    \normm{X_i(j) X_i(k)}_{\Psi_1} &\leq \normmo{X_i(j)} \normmo{X_i(k)}\\
    &\leq u \lambda_{\max}^1(\Sigma),
\end{align}
for all $i \in [n]$ and $(j,k) \in [p]\times [p]$. 

From Bernstein's Inequality (Theorem 2.8.1 in \citeauthor{Vershynin_2018} \citeyear{Vershynin_2018}), we thus have
    \begin{align}
        \PP \left[ | Y_{1,t,j,k}| \geq  c u \lambda_{\max}^1 (\Sigma) \left( \sqrt{\frac{x}{t}} \vee \frac{x}{t}\right) \right] &\leq 2e^{-x},
    \end{align}
    for $x>0$ and the same absolute constant $c\geq1$ as in Lemma \ref{covarianceconcentrationlemma}. Taking $x = 8 t c_6 c_0^{-1} \{ \log (ep)\vee \log\log(8n)\}$, where $c_0 = 16cu$ is the constant from Lemma \ref{covarianceconcentrationlemma}, we obtain
    \begin{align}
         \PP \left\{ | Y_{1,t,j,k}| > \frac{c_6}{2s} \lambda_{\max}^1(\Sigma) h(p,n,s,t)    \right\}  &\leq 2 \exp\left[- \frac{8c_6}{c_0}  \left\{ \log(ep) \vee \log \log(8n)\right\}\right]\\
         &\leq 2 \exp\left[- \frac{4c_6}{c_0}  \left\{ \log(ep) + \log \log(8n)\right\}\right]. 
    \end{align}
By a union bound, it follows that
\begin{align}
     \PP(\mathcal{E}_{10}^\complement)  &\leq  4 |\mathcal{T}| \ |\mathcal{S}|  p^2 \log(8n)^{- 4c_6 /c_0} (ep)^{-4 c_6/c_0} \exp \left\{ - \frac{4 c_6}{c_0} \log(ep) \right\}. 
     \intertext{Due to the definition of $c_6$, we have $4 c_6 / c_0 \geq 3$. It follows that }
      \PP(\mathcal{E}_{10}^\complement)  &\leq  4 e^{-3} {\frac{1 + \lfloor \log_2(n/2)\rfloor}{\log(8n)} } \log(8n)^{ 1- 4c_6 /c_0} p^{-1}  \left\{ 1 + \lfloor \log_2(p)\rfloor) \right\} \\
      &\leq  \log(8n)^{ 1- 4c_6 /c_0}\\
      &\leq \log(16)^{ 1- 4c_6 /c_0}\\
      &\leq \delta/4,
\end{align}
where we third last inequality used that $\sup_{p \geq 1} \ p^{-1} ( 1 + \lfloor \log_2 p \rfloor ) = 1$ and $\sup_{n \geq 2} \{ 1 + \lfloor \log_2 (n/2) \} / \log(8n) \leq 2$,  and the definition of $c_6$ in the last. The proof is complete.

\end{proof}

\begin{lemma}\label{eventbound6}
Assume $n \geq \log(ep)$ and let $X_1, \ldots, X_n$ be $p$-dimensional random variables satisfying Assumption \ref{assmultivariate2} for some $w,u>0$. For some $t \leq \lfloor n/2 \rfloor$, assume that $\EE X_i X_i^\top = \Sigma_1$ for $i \leq t$ , and $\EE X_i X_i^\top = \Sigma_2$ for $i \geq n - t+1$, where $\Sigma_1$ and $\Sigma_2$ are two positive definite matrices in $\RR^{p \times p}$. Let $\widehat{\Sigma}_{1,t} = t^{-1} \sum_{i=1}^t X_i X_i^\top$ and $\widehat{\Sigma}_{2,t} = t^{-1} \sum_{i=n-t+1}^n X_i X_i^\top$ and let
\begin{align}
    \widehat{\sigma}^2_{\mathrm{con}} &= \widehat{\lambda}^1_{\max} (\widehat{\Sigma}_{1,\lceil \log(ep)\rceil})  \wedge \widehat{\lambda}^1_{\max} (\widehat{\Sigma}_{2,\lceil \log(ep)\rceil}),
\end{align}
where $\widehat{\lambda}_{\max}^{s}(\cdot)$ is defined in \eqref{conveigdef}. For any fixed $\delta\in (0,1)$ and $s \in [p]$, define the events
    \begin{align}
    \mathcal{E}_{11} &=  \left\{\frac{\widehat{\sigma}^2_{\mathrm{con}}}{\lambda_{\max}^1(\Sigma_1) \wedge \lambda_{\max}^1(\Sigma_2)} \leq c_3\right\},\\
    \mathcal{E}_{12} &=  \bigcap_{i=1,2} \left\{ \widehat{\lambda}_{\max}^s ( {\widehat{\Sigma}_{i,t} - \Sigma_i}) \leq \frac{c_6}{2} \lambda_{\max}^1({\Sigma_i}) h(p,n,s,t) \right\},
    \end{align}
    where $h(p,n,s,t)$ is given in \eqref{hdef}, $\lambda_{\max}^s(\cdot)$ is defined in \eqref{firsteigsparse}, $c_3 = 1 + c_0 \log(1+4/\delta)$, $c_6 = c_0\{1 + \log(4/\delta)\}$, and $c_0$ is the constant from Lemma \ref{covarianceconcentrationlemma} depending only on $u>0$. Then $\PP\left( \mathcal{E}_{11} \cap \mathcal{E}_{12} \right)\geq 1- \delta/2$.
\end{lemma}
\begin{proof}
We first show that $\PP(\mathcal{E}_{11}^\complement)\leq \delta/4$. Without loss of generality, assume that $\lambda_{\max}^1(\Sigma_1) \leq  \lambda_{\max}^1(\Sigma_2)$. Observe that $\widehat{\lambda}_{\max}^1(A) = \lambda_{\max}^1(A) = \normm{A}_{\infty}$ for any positive definite matrix $A$. In particular, we have
\begin{align}
    \PP (\mathcal{E}_{11}^\complement) &\leq  \PP \left\{ {\lambda}^1_{\max} (\widehat{\Sigma}_{1,\lceil \log(ep)\rceil}) > c_3 \lambda_{\max}^1(\Sigma_1) \right\}. 
\end{align}
Using the same arguments as in the proof of Lemma \ref{eventbound4} it then follows that $\PP (\mathcal{E}_{11}^\complement) \leq \delta/4$. 
Moreover,  using the same arguments as in the proof of Lemma \ref{eventbound5}, we obtain $\PP(\mathcal{E}_{12}^\complement)\leq \delta/4$, 
   and the proof is complete. 
\end{proof}

\begin{lemma}\label{lemmachisqbound}
    Let $p\geq 1$ and $n\geq 2$. For $i = 1,2$, let $V_i\in \RR^{pn\times pn}$ be given by $V_i = \text{Diag}( \left\{ V_{j,i} \right\}_{j=1}^n )$, where $V_{j,i} = \sigma^2 I - \kappa_i u_i u_i^\top$ for $j \leq \Delta_i$ and $V_{j,i} = \sigma^2 I$ for $j>\Delta_i$, where $\sigma>0$, $1 \leq \Delta_i \leq n$, $0 < \kappa_i < \sigma^2$ and $u_i \in \mathcal{S}^{p-1}$ for $i=1,2$. Assume that $\Delta_1 \leq \Delta_2$. If $X \sim \text{N}(0, \sigma^2 I)$, then 
    \begin{align}
        \EE \left\{   \frac{\phi_{V_1}(X) \phi_{V_2}(X)}{\phi_{\sigma^2I}^2(X)}      \right\} &\leq \exp\left[   \frac{1}{2} \inner{u_1}{u_2}^2\left\{  \sqrt{\frac{\Delta_1}{\Delta_2}} (\Delta_2 \alpha_2^2)^{1/2} (\Delta_1 \alpha_1^2)^{1/2}  \wedge  \Delta_1 \alpha_1 \right\} \right],
    \end{align}
    where $\alpha_i = \kappa_i (\sigma^2 - \kappa_i)^{-1}$ for $i=1,2$.
\end{lemma}
\begin{proof}
We have that 
    \begin{align}
    & \quad \ \EE \left\{   \frac{\phi_{V_1}(X) \phi_{V_2}(X)}{\phi_{\sigma^2I}^2(X)}      \right\}\\
    &= \EE  \left[  \frac{\sigma^{2np}}{\det\left(V_1\right)^{1/2}\det\left(V_2\right)^{1/2}} \exp \left\{  -\frac{1}{2} X^\top \left( V_1^{-1} + V_2^{-1} - 2 \sigma^{-2} I \right)X  \right\}     \right].
\end{align}
Since $\det(V_i)=  \sigma^{2np}\left( 1 - \kappa_i /\sigma^2\right)^{\Delta_i} = \sigma^{2np} (1+ \alpha_i)^{\Delta_i}$, for $i=1,2$,  we obtain
\begin{align}
   \EE \left\{   \frac{\phi_{V_1}(X) \phi_{V_2}(X)}{\phi_{\sigma^2I}(X)}      \right\}&=   (1+\alpha_1)^{\Delta_1/2} (1+\alpha_2)^{\Delta_2/2} \\
   & \ \ \cdot \EE \left[   \exp \left\{  -\frac{1}{2} X^\top \left( V_1^{-1} + V_2^{-1} - 2 \sigma^{-2} I \right)X  \right\}     \right].
\end{align}
Since $V_i^{-1} = \text{Diag}\left(\left\{ V_{j,i}^{-1}\right\}_{j=1}^n\right)$ and $V_{j,i}^{-1} = \sigma^{-2} \left(I + \alpha_i u_i u_i^\top \ind\left \{j\leq \Delta_i\right\}\right)$, we obtain
\begin{align}
  \EE \left\{   \frac{\phi_{V_1}(X) \phi_{V_2}(X)}{\phi_{\sigma^2I}(X)}      \right\} &=(1+\alpha_1)^{\Delta_1/2} (1+\alpha_2)^{\Delta_2/2} \prod_{j=1}^n f_j, 
\end{align}
where
\begin{align}
    f_j &= \begin{cases} 1  & \text{ if } j > \Delta_2\\
    \EE   \ \exp \left(   - \frac{1}{2} \alpha_2 X^\top u_{2} u_{2}^\top X /\sigma^2 \right )  & \text{ if } \Delta_1 < j \leq \Delta_2 \\
    \EE \     \exp \left\{   - \frac{1}{2} X^\top \left( \alpha_1 u_{1} u_{1}^\top  + \alpha_2 u_{2} u_{2}^\top \right) X /\sigma^2\right \} & \text{ if }  j \leq \Delta_1.
    \end{cases}
\end{align}
For the case $\Delta_1 < j \leq \Delta_2$, we have that $f_j = \EE \exp \left( - \frac{\alpha_2}{2} Z^2\right) = (1 + \alpha_2)^{-1/2}$, where $Z \sim \text{N}(0,1)$.  For the case $1 \leq j \leq \Delta_1$ we have
\begin{align}
    I_j =& \EE\exp\left( - \frac{1}{2} \lambda_1 Z_1^2 - \frac{1}{2}\lambda_2 Z_2^2    \right),
\end{align}
where $Z_i \iid \text{N}(0,1)$ for $i=1,2$ and $\lambda_1, \lambda_2>0$ are the two (possibly) nonzero eigenvalues of the matrix $\alpha_1 u_{1} u_{1}^\top  + \alpha_2 u_{2} u_{2}^\top$, given by 
\begin{align}
    \lambda_1 &= \frac{\alpha_1 + \alpha_2 + \sqrt{(\alpha_1 - \alpha_2)^2 + 4\alpha_1 \alpha_2 \inner{u_1}{u_2}^2}}{2},\\
    \lambda_2 &= \frac{\alpha_1 + \alpha_2 - \sqrt{(\alpha_1 - \alpha_2)^2 + 4\alpha_1 \alpha_2 \inner{u_1}{u_2}^2}}{2}.
\end{align}
For $1\leq j \leq \Delta_1$, it thus follows  that
\begin{align}
    f_j &= \left( 1 + \lambda_1\right)^{-1/2}\left( 1 + \lambda_2\right)^{-1/2}\\
    &= \left\{ 1 + \alpha_1 + \alpha_2 + \alpha_1 \alpha_2 \left( 1- \inner{u_1}{u_2}^2 \right)  \right\}^{-1/2}.
\end{align}
We thus obtain
\begin{align}
   \EE \left\{   \frac{\phi_{V_1}(X) \phi_{V_2}(X)}{\phi_{\sigma^2I}(X)}      \right\}  &=  \left\{  \frac{(1+\alpha_1)(1+\alpha_2)}{1 + \alpha_1 + \alpha_2 + \alpha_1 \alpha_2 \left( 1- \inner{u_1}{u_2}^2 \right)}   \right\}^{\Delta_1/2}\\
    &\leq   \left(  1 + \frac{\alpha_1 \alpha_2 \inner{u_1}{u_2}^2}{1 + \alpha_1 + \alpha_2 }   \right)^{\Delta_1/2}\\
    &\leq  \exp \left( \frac{1}{2}\Delta_1 \frac{\alpha_1 \alpha_2 \inner{u_1}{u_2}^2}{1 + \alpha_1 + \alpha_2 }\right).\label{ineqrhs}
\end{align}
Observing that 

\begin{align}
    \Delta_1 \frac{\alpha_1 \alpha_2 }{1 + \alpha_1 + \alpha_2 } &\leq  \sqrt{\frac{\Delta_1}{\Delta_2}} \left(\Delta_2 \alpha_2\right)^{1/2} \left( \Delta_1 \alpha_1\right)^{1/2} \wedge  \Delta_1 \alpha_1,
\end{align}
the proof is complete. 
\end{proof}

\begin{lemma}\label{boundimplication}
Let  $n\geq 2$, $p \in \NN$ and $t_0\in [n-1]$. Let $\Sigma_1, \Sigma_2 \in \RR^{p\times p}$ be positive definite and symmetric matrices, and let $\sigma^2 = \normmop{\Sigma_1}\vee \normmop{\Sigma_2}$.  Define $\Delta = \Delta(t_0, n) = \min(t_0,n-t_0)$ and $\gamma = \gamma(p,n,s) =  s \log\left(\frac{ep}{s}\right) \vee \log \log(8n)$. Assume further that
\begin{align}
\Delta \left\{ \left( \frac{\lambda_{\max}^s (\Sigma_1 - \Sigma_2)}{\sigma^2 - \lambda_{\max}^s (\Sigma_1 - \Sigma_2)}  \right) \wedge \left( \frac{\lambda_{\max}^s (\Sigma_1 - \Sigma_2)}{\sigma^2 - \lambda_{\max}^s (\Sigma_1 - \Sigma_2)}\right)^2 \right\}&\geq c \gamma \label{necessarylemma1},
\end{align}
for some constant $c>0$, where $\lambda_{\max}^s(\cdot)$ is defined in \eqref{firsteigsparse}. Then we have
\begin{align}
    \underset{v \in \mathcal{S}^{p-1}_s}{\sup} \ \frac{v^\top \Sigma_1 v}{v^\top \Sigma_2 v} \vee \frac{v^\top \Sigma_2 v}{v^\top \Sigma_1 v} &\geq 1 + c \frac{ \gamma}{\Delta}\vee \sqrt{c \frac{ \gamma}{\Delta}}.\label{necessarylemma2}
\end{align}
In particular, Equation \eqref{necessarylemma2} is a necessary condition for Equation \eqref{necessarylemma1} to hold.
\end{lemma}
\begin{proof}
    Note first that \eqref{necessarylemma1} implies
    \begin{align}
        \lambda_{\max}^s (\Sigma_1 - \Sigma_2) \geq a \sigma^2 \label{necessaryproof1},
    \end{align}
    where $a = c\gamma (\Delta + c\gamma)^{-1} \vee (c\gamma)^{1/2} (\sqrt{\Delta} + \sqrt{c\gamma})^{-1}$. Since $S^{p-1}_s$ is a closed subset of the unit sphere $\mathcal{S}^{p-1}$, we can choose $v \in S^{p-1}_s$ such that $\lambda_{\max}^s (\Sigma_1 - \Sigma_2) = |v^\top (\Sigma_1 - \Sigma_2) v|$. Assume first that $v^\top (\Sigma_1 - \Sigma_2) v>0$. It then follows from Equation \eqref{necessaryproof1} that
        \begin{align}
       v^\top \Sigma_1 v  - v^\top \Sigma_2 v &\geq  a \sigma^2  .
        \intertext{In particular, since $\sigma^2 =\normmop{\Sigma_1}$ we have that}
        v^\top \Sigma_1 v  -a v^\top \Sigma_2 v &\geq a \normmop{\Sigma_1} \\
        &\geq a v^\top \Sigma_1 v.
        \intertext{Rearranging, we get}
        \frac{v^\top \Sigma_1 v}{v^\top \Sigma_2 v} &\geq \frac{1}{1-a}.
    \intertext{Similarly, if $v^\top (\Sigma_1 - \Sigma_2) v <0 $, we have}
        \frac{v^\top \Sigma_2 v}{v^\top \Sigma_1 v} &\geq \frac{1}{1-a} .
    \end{align}
    Inserting for $a$, we obtain the result. 
    
\end{proof}

In the following, we let $\mathrm{Sym}(p)$ denote the set of symmetric matrices in $\RR^{p\times p}$. We also let $\PSD(p)$ denote the set of positive semi-definite matrices in $\RR^{p\times p}$. 

The following Lemma is a formalization of results in Section 5 in \cite{bach2010convex}. 
\begin{lemma}\label{duallemma}
    Let $A \in \RR^{p\times p}$ be a symmetric matrix. Then we have
    \begin{align}
      \underset{\substack{Z \in N(p,s) }}{\sup} \ \mathrm{Tr}(AZ)  &\leq \underset{Y \in \mathrm{Sym}(p)}{\inf} \ \underset{\substack{Z \in \PSD(p)\\ \Tr(Z)=1}}{\sup} \Tr Z(A+Y) + s \normm{Y}_{\infty}\\
    &\leq  \underset{Y \in \mathrm{Sym}(p)}{\inf} \ \lambda_{\max}(A+Y) + s \normm{Y}_{\infty},
\end{align}
where $ N(p,s) = \{Z \in \PSD(p) \ ; \ \mathrm{Tr}(Z)=1, \normm{Z}_1\leq s\}$ and $\lambda_{\max}(\cdot) = \lambda_{\max}^p(\cdot)$ from \eqref{firsteigsparse}. 
\end{lemma}
Note that the result in \cite{bach2010convex} is slightly stronger as equalities prevail throughout in that paper. Since they do not prove the result,  we prove the slightly weaker Lemma \ref{duallemma} here. 
\begin{proof}
    We show the following chain of inequalities in steps: 
    \begin{align}
         \underset{\substack{Z \in N(p,s) }}{\sup} \ \mathrm{Tr}(AZ) &\leq \underset{\lambda\geq 0}{\inf} \ \underset{\substack{Z \in \PSD(p) \\ \Tr(Z)=1}}{\sup} \ \Tr(AZ)  + \lambda(s - \normm{Z}_1) \label{firsteq}\\
         &= \underset{\lambda\geq 0}{\inf} \  \underset{\substack{Z \in \PSD(p) \\ \Tr(Z)=1}}{\sup} \underset{\substack{Y \in \mathrm{Sym}(p)\\ \normm{Y}_{\infty}\leq 1}}{\inf} \ \Tr Z (A + \lambda Y) + s\lambda \label{secondeq}\\
         &\leq \underset{Y \in \mathrm{Sym}(p)}{\inf} \ \underset{\substack{Z \in \PSD(p) \\ \Tr(Z)=1}}{\sup} \Tr Z(A+Y) + s \normm{Y}_{\infty} \label{thirdeq}\\
         &=  \underset{Y \in \mathrm{Sym}(p)}{\inf} \ \lambda_{\max}(A + Y) + s \normm{Y}_{\infty} \label{fourtheq}.
    \end{align}

    \textbf{Step 1.} We begin by showing \eqref{firsteq}. For any $Z \in N(p,s)$ we have $Z \in \PSD(p)$, $\Tr(Z)=1$ and $\normm{Z}_1 \leq s$. It follows $\lambda(s - \normm{Z}_1)\geq 0$ for any $\lambda\geq 0$. Since we also have that $Z \in \PSD(p)$ and $\Tr(Z)=1$, we obtain
    \begin{align}
         \underset{\substack{Z \in N(p,s) }}{\sup} \ \mathrm{Tr}(AZ) &\leq \underset{\substack{Z \in \PSD(p) \\ \Tr(Z)=1}}{\sup} \ \Tr(AZ)  + \lambda(s - \normm{Z}_1),
    \end{align}
    for any $\lambda \geq 0$. Now take infimum over all $\lambda\geq 0$ on both sides to obtain \eqref{firsteq}.

    \textbf{Step 2.} To show \eqref{secondeq} it suffices to prove that 
    \begin{align}
        \Tr(AZ)  + \lambda(s - \normm{Z}_1) = \underset{\substack{Y \in \mathrm{Sym}(p)\\ \normm{Y}_{\infty}\leq 1}}{\inf} \ \Tr Z (A + \lambda Y) + s\lambda,\label{tempeq2}
    \end{align}
    for any $\lambda \geq 0$ and $Z \in \PSD(p)$ with $\Tr(Z)=1$. To this end, it suffices to show that 
    \begin{align}
        \underset{\substack{Y \in \mathrm{Sym}(p)\\ \normm{Y}_{\infty}\leq 1}}{\inf} \Tr(ZY) = - \normm{Z}_1.\label{tempeq3}
    \end{align}
    Note first that $ - \normm{Z}_1 = \Tr(Z\widehat{Y})$, where $\widehat{Y}$ is given by $\widehat{Y}_{i,j} = -\mathrm{sgn}(Z_{i,j})$. In particular, $\widehat{Y}$ is symmetric and $\lVert \widehat{Y}\rVert_{\infty}\leq 1$. This shows that the right hand side of \eqref{tempeq3} is no smaller than the left hand side. Conversely, for any $Y$ satisfying $\normm{Y}_{\infty}\leq 1$, we have
    \begin{align}
        \Tr(ZY) &= \sum_{i\in [p]}\sum_{j\in [p]} Z_{i,j}Y_{j_i}\\
        &\geq - \sum_{i\in [p]}\sum_{j\in [p]} |Z_{i,j}|\\
        &= -\normm{Z}_1.
    \end{align}
    Taking an infimum over all such $Y$, we obtain \eqref{tempeq3}.

    \textbf{Step 3.} We now show \eqref{thirdeq}. Fix any $\lambda \geq 0$. For any $Y \in \mathrm{Sym}(p)$ with $\normm{Y}_{\infty}=\lambda$, we have
    \begin{align}
        \underset{\substack{Z \in \PSD(p) \\ \Tr(Z)=1}}{\sup} \Tr Z(A+Y) + s \normm{Y}_{\infty} &=  \underset{\substack{Z \in \PSD(p) \\ \Tr(Z)=1}}{\sup} \Tr Z(A+Y) + s \lambda\\
        &=  \underset{\substack{Z \in \PSD(p) \\ \Tr(Z)=1}}{\sup} \Tr Z(A+\lambda Y) + s \lambda\\
        &\geq   \underset{\substack{Z \in \PSD(p) \\ \Tr(Z)=1}}{\sup}\ \underset{\substack{Y \in \mathrm{Sym}(p)\\ \normm{Y}_{\infty}\leq 1}}{\inf} \ \Tr Z(A+\lambda Y) + s \lambda.
    \end{align}
    It follows that 
    \begin{align}
        \underset{\substack{Y \in \mathrm{Sym}(p), \\ \normm{Y}_{\infty}=\lambda}}{\inf} \underset{\substack{Z \in \PSD(p) \\ \Tr(Z)=1}}{\sup} \Tr Z(A+Y) + s \normm{Y}_{\infty} &\geq \underset{\substack{Z \in \PSD(p) \\ \Tr(Z)=1}}{\sup} \ \underset{\substack{Y \in \mathrm{Sym}(p)\\ \normm{Y}_{\infty}\leq 1}}{\inf} \  \Tr Z(A+\lambda Y/\lambda) + s \lambda.
    \end{align}
    Taking an infimum over all $\lambda\geq 0$, we arrive at \eqref{thirdeq}.

    \textbf{Step 4.} The equality in \eqref{fourtheq} follows from the well known fact that
    \begin{align}
        \underset{\substack{Z \in \PSD(p) \\ \Tr(Z)=1}}{\sup} \Tr Z(Y) = \lambda_{\max} (Y) = \underset{v \in S^{p-1}}{\sup} \ v^\top Y v , \label{tempeq4}
    \end{align}
    for any symmetric $Y\in \RR^{p\times p}$, which we prove here for completeness. 
    
    Let $v \in S^{p-1}$. Then 
    \begin{align}
        v^\top Y v &= \Tr(v^\top Y v)\\
        &= \Tr(Y v v^\top).
    \end{align}
    Now let $Z = v v^\top$, which is positive semidefinite and satisfies $\Tr(Z) = \Tr(v v^\top) = v^\top v = 1$. Thus, 
    \begin{align}
        v^\top Y v \leq  \underset{\substack{Z \in \PSD(p) \\ \Tr(Z)=1}}{\sup} \ \Tr(YZ) 
    \end{align}
    Since $v$ was arbitrary, it follows that
    \begin{align}
        \underset{v \in S^{p-1}}{\sup} \ v^\top Y v \leq  \underset{\substack{Z \in \PSD(p) \\ \Tr(Z)=1}}{\sup} \ \Tr(YZ).
    \end{align}

    Conversely, let $Z \in \PSD(p)$ satisfy $\Tr(Z)=1$. Since $Z$ is symmetric, an eigendecomposition yields $Z = Q \Delta Q^\top$, where $Q = [q_1, \ldots, q_p]$ some matrix in $\RR^{p \times p}$ with orthonormal columns, and $\Delta = \mathrm{Diag}(\lambda_1(Z), \ldots, \lambda_p(Z))$ is a diagonal matrix with the eigenvalues of $Z$ on its diagonal. We have
    \begin{align}
        \Tr(ZY) &= \Tr(YZ)\\
        &= \Tr( Y Q\Delta Q^\top)\\
        &= \sum_{i=1}^p \lambda_i(Z)  \Tr (Y q_i q_i^\top)\\
        &= \sum_{i=1}^p \lambda_i(Z)  \Tr (q_i^\top Y q_i)\\
        &\leq \sum_{i=1}^p \lambda_i(Z)  \lambda_{\max}(Y)\\
        &\leq \lambda_{\max}(Y),
    \end{align}
    where we in the penultimate inequality used that $\lambda_i(Z)\geq 0$ for $i \in [p]$ and the fact that $\sum_{i=1}^p \lambda_i(Z) = \Tr(Z) = 1$. Taking a supremum, we obtain
    \begin{align}
        \underset{\substack{Z \in \PSD(p) \\ \Tr(Z)=1}}{\sup} \ \Tr(YZ) \leq \underset{v \in S^{p-1}}{\sup} \ v^\top Y v,
    \end{align}
    and we are done.

\end{proof}

\begin{lemma}\label{tullelemma}
Let $\Sigma \in \RR^{p\times p}$ be a symmetric positive definite matrix and let $\lambda_{\max}^s(\cdot)$ be defined as in \eqref{firsteigsparse}. Then $\lambda_{\max}^s(\Sigma)$ is non-decreasing in $s$, and for any $s_0/2 \leq s\leq s_0\leq p $ we have $\lambda_{\max}^{s_0}(\Sigma) \leq 4 \lambda_{\max}^{s}(\Sigma)$. Moreover, if $\Sigma_1, \Sigma_2 \in \RR^{p\times p}$ are two symmetric positive definite matrices, then $\lambda_{\max}^s(\Sigma_1  -\Sigma_2) \leq \lambda_{\max}^s(\Sigma_1) \vee \lambda_{\max}^s (\Sigma_2)$ for any $s \in [p]$.
\end{lemma}
\begin{proof}
    For the first claim,  note that $S^{p-1}_s$ is an increasing set in $s$. The first claim then follows immediately. For the second claim, note that since $\Sigma$ is symmetric and positive definite, we have
    \begin{align}
        \lambda_{\max}^{s_0}(\Sigma) &= \underset{v \in S_{s_0}^{p-1}}{\sup} \ |v^\top \Sigma v |\label{normmopdef}\\
        &= \underset{v \in S_{s_0}^{p-1}}{\sup} \ v^\top \Sigma v\\
        &=\underset{v \in S_{s_0}^{p-1}}{\sup}  \normm{  \Sigma^{1/2} v}_2^2
    \end{align}
    for all $s_0$, where $\Sigma^{1/2}$ is a symmetric positive definite matrix satisfying $\Sigma^{1/2} \Sigma^{1/2} = \Sigma$. It follows that
    \begin{align}
        \sqrt{\lambda_{\max}^{s_0}(\Sigma)} &= \underset{v \in S_{s_0}^{p-1}}{\sup}    \normm{  \Sigma^{1/2} v}_2.
    \end{align}
    For any $v \in S_{s_0}^{p-1}$, we may write $v = v_1 + v_2$, where $v_1, v_2 \in S_s^{p-1}$ since $s \geq s_0/2$. Hence for any $v \in S_{s_0}^{p-1}$, it follows that
    \begin{align}
         \normm{  \Sigma^{1/2} v}_2 &\leq   \normm{  \Sigma^{1/2} v_1}_2 +  \normm{  \Sigma^{1/2} v_2}_2\\
         &\leq 2 \underset{v_s \in S_{s}^{p-1}}{\sup}  \normm{  \Sigma^{1/2} v_s}_2 \\
         &= 2 \sqrt{\lambda_{\max}^s(\Sigma)}.
    \end{align}
    As $v$ was arbitrary, the second claim follows. 

    For the third claim, note that 
    \begin{align}
        \lambda_{\max}^s (\Sigma_1 - \Sigma_2) &= \underset{s \in S^{p-1}_s}{\sup} | v^\top (\Sigma_1 - \Sigma_2)v|\\
        &= \underset{s \in S^{p-1}_s}{\sup}  \left\{  v^\top (\Sigma_1 - \Sigma_2)v  \vee v^\top (\Sigma_2 - \Sigma_1)v \right\}.
    \end{align}
    Since $\Sigma_1$ and $\Sigma_2$ are both positive definite, both $v^\top \Sigma_1 v$ and $v^\top \Sigma_2 v$ are positive. It follows that 
    \begin{align}
        \lambda_{\max}^s (\Sigma_1 - \Sigma_2) &\leq \underset{s \in S^{p-1}_s}{\sup}  \left\{  v^\top\Sigma_1v  \vee v^\top \Sigma_2 v \right\}
        \\ &= \lambda_{\max}^s(\Sigma_1) \vee \lambda_{\max}^s (\Sigma_2),
    \end{align}
    and the proof is complete.
\end{proof}

\begin{lemma}[\citeauthor{liu_minimax_2021} \citeyear{liu_minimax_2021}, Supplementary Material, Lemma 8]\label{minimaxlemma}
Let $\Theta_0$ and $\Theta_1$ denote general parameter spaces, and consider a family of distributions $\{\PP_{\theta}\}_{\theta \in \Theta}$ on $\RR^p$, where $\Theta = \Theta_0 \cup \Theta_1$. Let $\nu_0$ and $\nu_1$ be two distributions supported on $\Theta_0$ and $\Theta_1$ respectively. For $r \in \{0,1\}$, define $\mathbb{Q}_r$ to be the marginal distribution of the random variable $X$ generated hierarchically according to $\theta \sim \nu_r$ and $X|\theta \sim \PP_{\theta}$. Then
\begin{align}
    \underset{\psi \in \Psi}{\inf} \left[   \underset{\theta \in \Theta_0}{\sup} \ \EE_{\theta} \psi(X) + \underset{\theta \in \Theta_1}{\sup}\ \EE_{\theta} \left\{ 1- \psi(X)\right\}     \right] &\geq \max \left\{ \frac{1}{2}\exp(-\alpha), 1 - \sqrt{\frac{\alpha}{2}} \right\},
\end{align}
where $\Psi$ is the set of measurable functions $\psi : \RR^p \mapsto \{0,1\}$ and $\alpha = \chi^2(\mathbb{Q}_0 \ \| \ \mathbb{Q}_1)$.
    
\end{lemma}

\begin{lemma}[\citeauthor{cai_optimal_2015} \citeyear{cai_optimal_2015}, Lemma 1]\label{rademacherlemma}
Let $p \in \NN$ and $s \in [p]$. Let $R_1, \ldots, R_s$ be independently Rademacher distributed. Denote the symmetric random walk on $\ZZ$ stopped at the $h$-th step by 
\begin{align}
    G(h) = \sum_{i=1}^h R_i.
\end{align}
Let $H\sim \text{Hyp}(p,s,s)$ be independent of $\{R_i\}_{i \in [s]}$. Then there exists a function $g \ : \ (0, 1/36) \mapsto (1, \infty)$ with ${\lim}_{x \downarrow 0 } \ {g(x)} = 1$ such that for any $a < 1/36$, 
\begin{align}
\EE \exp \left\{   \frac{a}{s}\log \left( \frac{ep}{s}  \right)  G^2(H)  \right\}\leq g(a).
\end{align}
\end{lemma}

\section{Acknowledgments}
The author wishes to express sincere appreciation to Ingrid Kristine Glad and Martin Tveten for their feedback and enduring support.

\bibliography{main}

\end{document}